\documentclass{amsart}
\usepackage{graphicx}
\usepackage[all]{xy}
\usepackage{amsthm}
\usepackage{amssymb}
\usepackage[nospace,noadjust]{cite}
\usepackage{mathtools}
\usepackage{romannum}
\usepackage{tikz-cd}
\usepackage[english]{babel}
\usepackage[autostyle]{csquotes}

\usepackage{amsmath,amsfonts}

\theoremstyle{definition}
\newtheorem{1def}{Definition}[section]

\newtheorem{note}[1def]{Note}

\theoremstyle{plain}
\newtheorem{thm}[1def]{Theorem}
\newtheorem{lem}[1def]{Lemma}
\newtheorem{pro}[1def]{Proposition}
\newtheorem{cor}[1def]{Corollary}

\theoremstyle{remark}
\newtheorem{rmk}[1def]{Remark}

\numberwithin{equation}{section}

\begin{document}
\title[$\square_q$-modules and their Drinfel'd polynomials]{Finite-dimensional irreducible $\square_q$-modules\\ and their Drinfel'd polynomials}
\author{Yang Yang}
\address{Department of Mathematics, University of Wisconsin, Madison, WI 53706, USA}
\email{yyang@math.wisc.edu}

\begin{abstract}
Let $\mathbb{F}$ denote an algebraically closed field with characteristic $0$, and let $q$ denote a nonzero scalar in $\mathbb{F}$ that is not a root of unity. Let $\mathbb{Z}_4$ denote the cyclic group of order $4$. Let $\square_q$ denote the unital associative $\mathbb{F}$-algebra defined by generators $\{x_i\}_{i\in \mathbb{Z}_4}$ and relations
\begin{gather*}
\frac{qx_ix_{i+1}-q^{-1}x_{i+1}x_i}{q-q^{-1}}=1, \\
x_i^3x_{i+2}-[3]_qx_i^2x_{i+2}x_i+[3]_qx_ix_{i+2}x_i^2-x_{i+2}x_i^3=0,
\end{gather*}
where $[3]_q=(q^3-q^{-3})/(q-q^{-1})$. There exists an automorphism $\rho$ of $\square_q$ that sends $x_i\mapsto x_{i+1}$ for $i\in \mathbb{Z}_4$. Let $V$ denote a finite-dimensional irreducible $\square_q$-module of type $1$. To $V$ we attach a polynomial called the Drinfel'd polynomial. In our main result, we explain how the following are related:
\begin{enumerate}
\item[\rm(i)] the Drinfel'd polynomial for the $\square_q$-module $V$;
\item[\rm(ii)] the Drinfel'd polynomial for the $\square_q$-module $V$ twisted via $\rho$.
\end{enumerate}
Specifically, we show that the roots of (i) are the inverses of the roots of (ii). We discuss how $\square_q$ is related to the quantum loop algebra $U_q(L(\mathfrak{sl}_2))$, its positive part $U_q^+$, the $q$-tetrahedron algebra $\boxtimes_q$, and the $q$-geometric tridiagonal pairs.


\end{abstract}
\maketitle
\pagenumbering{arabic}
\section{Introduction}
Throughout the paper let $\mathbb{F}$ denote an algebraically closed field with characteristic $0$. We will discuss a number of algebras over $\mathbb{F}$. In \cite{b1} B. Hartwig and P. Terwilliger introduced the tetrahedron Lie algebra $\boxtimes$. In \cite[\rm Theorem~11.5]{b1} they showed that $\boxtimes$ is isomorphic to the three-point $\mathfrak{sl}_2$ loop algebra. In \cite{b2} B. Hartwig classified up to isomorphism the finite-dimensional irreducible $\boxtimes$-modules. For more information about $\boxtimes$ see \cite{b3,b4}. Fix $0\ne q\in \mathbb{F}$ such that $q$ is not a root of unity. In \cite{qtet} T. Ito and P. Terwilliger introduced a quantum analog $\boxtimes_q$ of $\boxtimes$ called the $q$-tetrahedron algebra. In \cite[\rm Section~10]{qtet} they classified up to isomorphism the finite-dimensional irreducible $\boxtimes_q$-modules. For more information about $\boxtimes_q$ see \cite{qtet2, qtet3, qtet4, qtet5, qtet6, miki}. We now mention a chain of algebra homomorphisms that involves $\boxtimes_q$. Consider the quantum loop algebra $U_q(L(\mathfrak{sl}_2))$ \cite[Section~3.3]{cp} and its positive part $U_q^+$ \cite[\rm Definition~1.1]{non}. By \cite[\rm Propositions~4.1, 4.3]{miki}, there exists a chain of algebra homomorphisms
\begin{equation}
\begin{tikzcd}
 U_q^+\arrow{r} & U_q(L(\mathfrak{sl}_2)) \arrow{r} & \boxtimes_q. \label{c1}
\end{tikzcd}
\end{equation}
In the above chain each map is injective. In \cite{tri} T. Ito, K. Tanabe and P. Terwilliger introduced the notion of a tridiagonal pair. In \cite{cla} T. Ito, K. Nomura and P. Terwilliger classified up to isomorphism the tridiagonal pairs over $\mathbb{F}$. In \cite[\rm Theorems~10.3, 10.4]{qtet} and \cite[\rm Theorem~2.7]{non}, the finite-dimensional irreducible $\boxtimes_q$-modules are linked to a family of tridiagonal pairs said to have $q$-geometric type. For more information about tridiagonal pairs see \cite{stru,dri,uqsl2hat,h1}. The existing literature contains a bijection between any two of the following sets:
\begin{enumerate}
\item[\rm(i)] the isomorphism classes of $q$-geometric tridiagonal pairs;
\item[\rm(ii)] the isomorphism classes of NonNil finite-dimensional irreducible $U_q^+$-modules of type $(1,1)$;
\item[\rm(iii)] the isomorphism classes of finite-dimensional irreducible $U_q(L(\mathfrak{sl}_2))$-modules $V$ of type $1$ such that $P_V(1)\ne 0$, where $P_V$ is the Drinfel'd polynomial of $V$;
\item[\rm(iv)] the isomorphism classes of finite-dimensional irreducible $\boxtimes_q$-modules of type $1$;
\item[\rm(v)] the polynomials in $\mathbb{F}[z]$ that have constant coefficient $1$ and do not vanish at $z=1$.
\end{enumerate}
The bijection between the sets (i), (ii) is established in \cite[\rm Lemma~4.8]{non}. The bijection between the sets (ii), (iii) is established in \cite[\rm Theorems~1.6, 1.7]{non} by using the map $U_q^+\to U_q(L(\mathfrak{sl}_2))$ in the chain (\ref {c1}). The bijection between the sets (ii), (iv) is established in \cite[\rm Theorems~10.3, 10.4]{qtet} by using the map $U_q^+\to \boxtimes_q$ in the chain (\ref {c1}). The bijection between the sets (iii), (v) is established in \cite[\rm p.~261]{cp}.

In \cite{p1} P. Terwilliger introduced a variation of $\boxtimes_q$ called $\square_q$. By \cite[\rm Propositions~4.1, 4.3]{miki} and \cite[\rm Proposition~5.5]{p1}, the chain (\ref {c1}) has an extension
\begin{equation}
\begin{tikzcd}
 U_q^+\arrow{r} & \square_q \arrow{r} & U_q(L(\mathfrak{sl}_2)) \arrow{r} & \boxtimes_q. \label{c2}
\end{tikzcd}
\end{equation}
In the above chain each map is injective.

In the present paper we focus on the finite-dimensional irreducible $\square_q$-modules. We have two main results. Our first main result is about how the isomorphism classes of finite-dimensional irreducible $\square_q$-modules are related to the sets (i)--(v). We identify a class of finite-dimensional irreducible $\square_q$-modules said to have type $1$. Using the chain (\ref {c2}) we show that the above sets (i)--(v) are in bijection with (vi) the isomorphism classes of finite-dimensional irreducible $\square_q$-modules of type $1$. Our second main result is about what happens when we twist a finite-dimensional irreducible $\square_q$-module of type $1$, via a certain automorphism $\rho$ of $\square_q$ that has order $4$.

We next describe our main results in more detail. We begin with some notation and basic concepts. Recall the ring of integers $\mathbb{Z}=\{0,\pm 1,\pm 2,\dots\}$. Let $\mathbb{Z}_4=\mathbb{Z}/4\mathbb{Z}$ denote the cyclic group of order $4$. Let $\mathcal{A}$ denote an $\mathbb{F}$-algebra and let $\varphi: \mathcal{A}\to \mathcal{A}$ denote an $\mathbb{F}$-algebra homomorphism. Let $V$ denote an $\mathcal{A}$-module. There exists an $\mathcal{A}$-module structure on $V$, called $V$ {\it twisted via} $\varphi$, that behaves as follows: for all $a\in \mathcal{A}$ and $v\in V$, the vector $av$ computed in $V$ twisted via $\varphi$ coincides with the vector $\varphi(a)v$ computed in the original $\mathcal{A}$-module $V$. Sometimes we abbreviate $^{\varphi}V$ for $V$ twisted via $\varphi$. Let $z$ denote an indeterminate and let $\mathbb{F}[z]$ denote the $\mathbb{F}$-algebra consisting of the polynomials in $z$ that have all coefficients in $\mathbb{F}$. We now recall the algebra $\square_q$.

\begin{1def}
\label{def:box}
(See \cite[Definition~5.1]{p1}.)
Let $\square_q$ denote the $\mathbb{F}$-algebra with generators $\{x_i\}_{i\in \mathbb{Z}_4}$ and relations
\begin{gather}
\frac{qx_ix_{i+1}-q^{-1}x_{i+1}x_i}{q-q^{-1}}=1, \label{equ:11}\\
x_i^3x_{i+2}-[3]_qx_i^2x_{i+2}x_i+[3]_qx_ix_{i+2}x_i^2-x_{i+2}x_i^3=0,\label{equ:12}
\end{gather}
where $[3]_q=(q^3-q^{-3})/(q-q^{-1})$.
\end{1def}

\begin{lem}
\label{lem:rho}
There exists an automorphism $\rho$ of $\square_q$ that sends $x_i\mapsto x_{i+1}$ for $i\in \mathbb{Z}_4$. Moreover $\rho^4=1$.
\end{lem}
\begin{proof}
By Definition {\ref {def:box}}.
\end{proof}

\begin{lem}
\label{lem:gam}
For nonzero $\alpha\in\mathbb{F}$ there exists an automorphism of $\square_q$ that sends
\begin{gather*}
x_0\mapsto \alpha^{-1}x_0,\qquad x_1\mapsto \alpha x_1,\\
x_2\mapsto \alpha^{-1}x_2,\qquad x_3\mapsto \alpha x_3.
\end{gather*}
\end{lem}
\begin{proof}
By Definition {\ref {def:box}}.
\end{proof}

We have some comments about $\square_q$-modules. Let $V$ denote a finite-dimensional irreducible $\square_q$-module. We will show that each generator $x_i$ of $\square_q$ is semisimple on $V$. Moreover there exist an integer $d\ge 0$ and nonzero $\gamma\in\mathbb{F}$ with the following property:
\begin{enumerate}
\item[\rm(i)] for $x_0$ and $x_2$, the set of distinct eigenvalues on $V$ is $\{\gamma q^{d-2i}\}_{i=0}^{d}$;
\item[\rm(ii)] for $x_1$ and $x_3$, the set of distinct eigenvalues on $V$ is $\{\gamma^{-1}q^{d-2i}\}_{i=0}^{d}$.
\end{enumerate}
We call $d$ the {\it diameter} of $V$, and call $\gamma$ the {\it type} of $V$. Observe that the $\square_q$-module $^{\rho}V$ has diameter $d$ and type $\gamma^{-1}$. Next we twist the $\square_q$-module $V$ via the automorphism from Lemma {\ref {lem:gam}}. The resulting $\square_q$-module has diameter $d$ and type $\gamma\alpha^{-1}$. In particular, if $\alpha=\gamma$ then this $\square_q$-module has type $1$. Motivated by the above comments, we focus on the finite-dimensional irreducible $\square_q$-modules of type $1$.

Let $V$ denote a finite-dimensional irreducible $\square_q$-module of type $1$. We will associate to $V$ a polynomial $P_{V}\in\mathbb{F}[z]$ that has constant coefficient $1$. This $P_{V}$ depends on the isomorphism class of the $\square_q$-module $V$. We call $P_{V}$ the {\it Drinfel'd polynomial of $V$}. We will compare the Drinfel'd polynomials for $V$ and $^{\rho}V$.

The significance of $P_{V}$ is indicated in the following proposition, which is our first main result.

\begin{pro}
\label{thm:36}
The map $V\mapsto P_V$ induces a bijection between the following two sets:
\begin{enumerate}
\item[\rm(i)] the isomorphism classes of finite-dimensional irreducible $\square_q$-modules of type $1$;
\item[\rm(ii)] the polynomials in $\mathbb{F}[z]$ that have constant coefficient $1$ and do not vanish at $z=1$.
\end{enumerate}
\end{pro}

Before stating our second main result, we have some comments.
\begin{1def}
\label{def:part}
Pick any $f\in\mathbb{F}[z]$ with constant coefficient $1$. Write
\begin{gather*}
f=a_0+a_1z+\dots+a_dz^d,\qquad a_0=1,\quad a_d\ne 0.
\end{gather*}
For $0\le i\le d$ replace $a_i$ by $a_{d-i}/a_d$ to get a polynomial
\begin{gather*}
f^\vee=\frac{a_d+a_{d-1}z+\dots+a_0z^d}{a_d}.
\end{gather*}
Observe that $f^{\vee}$ has constant coefficient $1$ and $(f^{\vee})^{\vee}=f$. We say that $f,f^{\vee}$ are {\it partners}. Note that
\begin{gather*}
f^{\vee}(z)=\frac{z^df(z^{-1})}{a_d}.
\end{gather*}
Let $z_1,z_2,\dots,z_d$ denote the roots of $f$. Then $z_1^{-1},z_2^{-1},\dots,z_d^{-1}$ are the roots of $f^{\vee}$.
\end{1def}

We now state our second main result.
\begin{thm}
\label{thm:37}
Let $V$ denote a finite-dimensional irreducible $\square_q$-module of type $1$. Then the Drinfel'd polynomials for $V$ and $^\rho V$ are partners.
\end{thm}

\begin{cor}
\label{cor:38}
Let $V$ denote a finite-dimensional irreducible $\square_q$-module of type $1$. Then the $\square_q$-modules $V$ and $^{\rho^2} V$ are isomorphic.
\end{cor}

We now mention the idea behind our proof of Proposition {\ref {thm:36}}. Using the map $U_q^+\to \square_q$ in the chain (\ref {c2}), we will establish a bijection between the isomorphism classes of finite-dimensional irreducible $\square_q$-modules of type $1$ and the isomorphism classes of NonNil finite-dimensional irreducible $U_q^+$-modules of type $(1,1)$. In {\cite{non}}, the NonNil finite-dimensional irreducible $U_q^+$-modules of type $(1,1)$ are classified using the Drinfel'd polynomial of $U_q(L(\mathfrak{sl}_2))$. Using these facts we establish the bijection in Proposition {\ref {thm:36}}.

We now mention the idea behind our proof of Theorem {\ref {thm:37}}. Let $V$ denote a finite-dimensional irreducible $\square_q$-module of type $1$. We will show that for $i\in\mathbb{Z}_4$ the $\square_q$ generators $x_i,x_{i+2}$ act on $V$ as a $q$-geometric tridiagonal pair (see \cite[\rm Definition~2.6]{non}). Using the Drinfel'd polynomials for these tridiagonal pairs (see \cite{dri}) together with some results from \cite{cp}, we prove Theorem {\ref {thm:37}}.

Given the earlier literature, Theorem {\ref {thm:37}} and Corollary {\ref {cor:38}} should not be surprising. Concerning Theorem {\ref {thm:37}} we mention one item \cite[\rm Lemma~9.11]{qtet6}. This is about the algebra $\boxtimes_q$. Let $V$ denote an evaluation module for $\boxtimes_q$ with evaluation parameter $t$ (see \cite[\rm Section~9]{qtet6}). Let $\rho$ denote the automorphism of $\boxtimes_q$ from \cite[\rm Lemma~6.3]{qtet}. It was shown in \cite[\rm Lemma~9.11]{qtet6} that the $\boxtimes_q$-module $^{\rho}V$ is an evaluation module with evaluation parameter $t^{-1}$. Concerning Corollary {\ref {cor:38}} we mention one item \cite[\rm Theorem~2.5]{h1}. This is about $q$-geometric tridiagonal pairs. Let $V$ denote a vector space over $\mathbb{F}$ with finite positive dimension, and let $A,A^*$ denote a $q$-geometric tridiagonal pair on $V$. It was shown in \cite[\rm Theorem~2.5]{h1} that the tridiagonal pairs $A,A^*$ and $A^*,A$ are isomorphic.

The paper is organized as follows. Section 2 contains the preliminaries. Section 3 contains some basic facts about $U_q(L(\mathfrak{sl}_2))$. Section 4 contains some basic facts about Drinfel'd polynomials for $U_q(L(\mathfrak{sl}_2))$-modules. Section 5 contains some basic facts about $U_q^+$. In Sections 6--8 we describe the finite-dimensional irreducible $\square_q$-modules and prove Proposition {\ref {thm:36}}. Section 9 contains some basic facts about tridiagonal pairs and their Drinfel'd polynomial. In Section 10 we prove Theorem {\ref {thm:37}}. In Section 11 we obtain an analogous result for $\boxtimes_q$.

\section{Preliminaries}
We now begin our formal argument. Throughout this paper, every algebra without the Lie prefix is meant to be associative and have a $1$. Let $\mathcal{A}$ denote an $\mathbb{F}$-algebra. By an {\it automorphism} of $\mathcal{A}$ we mean an $\mathbb{F}$-algebra isomorphism from $\mathcal{A}$ to $\mathcal{A}$. Let $\mathcal{B}$ denote an $\mathbb{F}$-algebra and $\varphi: \mathcal{A}\to \mathcal{B}$ denote an $\mathbb{F}$-algebra homomorphism. For a $\mathcal{B}$-module $V$, we pull back the $\mathcal{B}$-module structure on $V$ via $\varphi$ and turn $V$ into an $\mathcal{A}$-module. We call the $\mathcal{A}$-module $V$ the {\it pullback} of the $\mathcal{B}$-module $V$ via $\varphi$. If $\varphi$ is surjective, then the $\mathcal{B}$-module $V$ is irreducible if and only if the $\mathcal{A}$-module $V$ is irreducible. We mention a special case in which $\mathcal{A}=\mathcal{B}$. In this case the following are the same: (i) the pullback of the $\mathcal{B}$-module $V$ via $\varphi$; (ii) the $\mathcal{B}$-module $V$ twisted via $\varphi$.

Recall the set of natural numbers $\mathbb{N}=\{0,1,2,\dots\}$. For the duration of this paragraph fix $d\in\mathbb{N}$. Let $V$ denote a vector space over $\mathbb{F}$ with dimension $d+1$. Let ${\rm End}(V)$ denote the $\mathbb{F}$-algebra consisting of the $\mathbb{F}$-linear maps from $V$ to $V$. Let $I\in {\rm End}(V)$ denote the identity map. Let $L\in {\rm End}(V)$. The map $L$ is said to be {\it nilpotent} whenever there exists a positive integer $n$ such that $L^n=0$. For $\theta\in\mathbb{F}$ define
\begin{gather*}
V_L(\theta)=\{v\in V\mid Lv=\theta v\}.
\end{gather*}
Observe that $\theta$ is an eigenvalue of $L$ if and only if $V_L(\theta)\ne 0$; in this case $V_L(\theta)$ is the corresponding eigenspace. We say that $L$ is {\it semisimple} whenever $V$ is spanned by the eigenspaces of $L$. Let ${\rm{Mat}}_{d+1}(\mathbb{F})$ denote the $\mathbb{F}$-algebra consisting of the $d+1$ by $d+1$ matrices that have all entries in $\mathbb{F}$. We index the rows and columns by $0,1,\dots,d$. Let $\{v_i\}_{i=0}^d$ denote a basis for $V$. For $A\in {\rm End}(V)$ and $M\in{\rm{Mat}}_{d+1}(\mathbb{F})$, we say that {\it $M$ represents $A$ with respect to $\{v_i\}_{i=0}^d$} whenever $Av_j=\sum_{i=0}^d M_{ij}v_i$ for $0\le j\le d$. For $n\in \mathbb{Z}$ define
\begin{equation*}
[n]_q = \frac{q^n-q^{-n}}{q-q^{-1}}.
\end{equation*}
For $n\in \mathbb{N}$ define
\begin{equation*}
[n]_q^! = \prod_{i=1}^{n}[i]_q.
\end{equation*}
We interpret $[0]_q^!=1$.

\section{The quantum loop algebra $U_q(L(\mathfrak{sl}_2))$}
In this section, we recall the quantum loop algebra $U_q(L(\mathfrak{sl}_2))$.
\begin{1def}
\label{def:lp1}
(See \cite[Section~3.3]{cp}.)
Let $U_q(L(\mathfrak{sl}_2))$ denote the $\mathbb{F}$-algebra with generators $e_i^{\pm}$, $K_{i}^{\pm 1}$, $i\in \{0,1\}$ and relations
\begin{gather*}
K_iK_i^{-1}=K_i^{-1}K_i=1,\\
K_0K_1=K_1K_0=1,\\
K_ie_i^{\pm}K_i^{-1}=q^{\pm 2}e_i^{\pm},\\
K_ie_j^{\pm}K_i^{-1}=q^{\mp 2}e_j^{\pm}, \qquad i\ne j,\\
e_i^{+}e_i^{-}-e_i^{-}e_i^{+}=\frac{K_i-K_i^{-1}}{q-q^{-1}},\\
e_0^{\pm}e_1^{\mp}=e_1^{\mp}e_0^{\pm},\\
(e_i^{\pm})^3e_j^{\pm}-[3]_q(e_i^{\pm})^2e_j^{\pm}e_i^{\pm}+[3]_qe_i^{\pm}e_j^{\pm}(e_i^{\pm})^2-e_j^{\pm}(e_i^{\pm})^3=0, \qquad i\ne j.
\end{gather*}
We call $e_i^{\pm}$, $K_{i}^{\pm 1}$, $i\in \{0,1\}$ the {\it Chevalley generators} for $U_q(L(\mathfrak{sl}_2))$.
\end{1def}

In a moment we will discuss some objects $X_{ij}$. The subscripts
$i,j$ are meant to be in $\mathbb{Z}_4$. We now recall the equitable presentation for $U_q(L(\mathfrak{sl}_2))$.
\begin{lem}
\label{lem:lp2}
\rm (See \cite[\rm Theorem~2.1]{uqsl2hat},
\cite[\rm Proposition~4.2]{miki}.)
\it
The $\mathbb{F}$-algebra $U_q(L(\mathfrak{sl}_2))$ is isomorphic to the $\mathbb{F}$-algebra with generators
\begin{gather}
X_{01}, \quad X_{12},\quad X_{23},\quad X_{30}, \quad X_{13},\quad X_{31} \label{41}
\end{gather}
and the following relations:
\begin{gather*}
X_{13}X_{31}=X_{31}X_{13}=1,\\
\frac{qX_{01}X_{13}-q^{-1}X_{13}X_{01}}{q-q^{-1}}=1,\quad\qquad \frac{qX_{13}X_{30}-q^{-1}X_{30}X_{13}}{q-q^{-1}}=1,\\
\frac{qX_{23}X_{31}-q^{-1}X_{31}X_{23}}{q-q^{-1}}=1,\quad\qquad \frac{qX_{31}X_{12}-q^{-1}X_{12}X_{31}}{q-q^{-1}}=1;
\end{gather*}
\begin{gather*}
\frac{qX_{i,i+1}X_{i+1,i+2}-q^{-1}X_{i+1,i+2}X_{i,i+1}}{q-q^{-1}}=1\qquad (i\in \mathbb{Z}_4),\\
X_{i,i+1}^3X_{i+2,i+3}-[3]_qX_{i,i+1}^2X_{i+2,i+3}X_{i,i+1}+[3]_qX_{i,i+1}X_{i+2,i+3}X_{i,i+1}^2\\
-X_{i+2,i+3}X_{i,i+1}^3=0\qquad (i\in \mathbb{Z}_4).
\end{gather*}
An isomorphism sends
\begin{gather}
X_{01}\mapsto K_0+q(q-q^{-1})K_0e^{-}_0, \qquad X_{23}\mapsto K_1+q(q-q^{-1})K_1e^{-}_1, \label{42}\\
X_{12}\mapsto K_1-(q-q^{-1})e^{+}_1, \qquad X_{30}\mapsto K_0-(q-q^{-1})e^{+}_0, \label{43}\\
X_{13}\mapsto K_1,\qquad X_{31}\mapsto K_0. \label{44}
\end{gather}
The inverse isomorphism sends
\begin{gather*}
e_0^{-}\mapsto q^{-1}(q-q^{-1})^{-1}(X_{13}X_{01}-1),\qquad e_1^{-}\mapsto q^{-1}(q-q^{-1})^{-1}(X_{31}X_{23}-1),\\
e_1^{+}\mapsto (q-q^{-1})^{-1}(X_{13}-X_{12}),\qquad e_0^{+}\mapsto (q-q^{-1})^{-1}(X_{31}-X_{30}),\\
K_1\mapsto X_{13},\qquad K_0\mapsto X_{31}.
\end{gather*}
\end{lem}

\begin{note}
For notational convenience, we identify the copy of $U_q(L(\mathfrak{sl}_2))$ given in Definition {\ref {def:lp1}} with the copy given in Lemma {\ref {lem:lp2}}, via the isomorphism given in Lemma {\ref {lem:lp2}}.
\end{note}

\begin{1def}
We call the elements ({\ref {41}}) the {\it equitable generators} for $U_q(L(\mathfrak{sl}_2))$.
\end{1def}

Comparing the relations in Definition {\ref {def:box}} with the relations in Lemma {\ref {lem:lp2}}, we obtain an $\mathbb{F}$-algebra homomorphism $\psi: \square_q \to U_q(L(\mathfrak{sl}_2))$ that sends $x_i\mapsto X_{i,i+1}$ for $i\in\mathbb{Z}_4$.
\begin{lem}
\rm
(See \cite[Proposition~5.5]{p1}, \cite[Propositions~4.1, 4.3]{miki}.)
\label{lem:1}
\it
The above homomorphism $\psi$ is injective.
\end{lem}

We next recall some facts about finite-dimensional irreducible $U_q(L(\mathfrak{sl}_2))$-modules.

We begin with some notation. Let $V$ denote a vector space over $\mathbb{F}$ with finite positive dimension. Let $\{s_i\}_{i=0}^{d}$ denote a finite sequence of positive integers whose sum is the dimension of $V$. By a {\it decomposition} of $V$ of {\it shape} $\{s_i\}_{i=0}^{d}$ we mean a sequence $\{U_i\}_{i=0}^{d}$ of subspaces of $V$ such that $U_i$ has dimension $s_i$ for $0\le i\le d$, and the sum $V=\sum_{i=0}^d U_i$ is direct. We call $d$ the {\it diameter} of the decomposition. For $0\le i\le d$ we call $U_i$ the $i$th {\it component} of the decomposition. For notational convenience define $U_{-1}=0$ and $U_{d+1}=0$. By the {\it inversion} of $\{U_i\}_{i=0}^{d}$ we mean the decomposition $\{U_{d-i}\}_{i=0}^{d}$.
\begin{lem}
\label{lem:dec}
\rm
(See \cite[\rm Proposition~3.2]{cp}.)
\it
Let $V$ denote a finite-dimensional irreducible $U_q(L(\mathfrak{sl}_2))$-module. Then there exist a unique scalar $\gamma\in\{1,-1\}$ and a unique decomposition $\{U_i\}_{i=0}^{d}$ of $V$ such that
\begin{gather*}
(K_0-\gamma q^{2i-d}I)U_i=0,\qquad (K_1-\gamma q^{d-2i}I)U_i=0
\end{gather*}
for $0\le i\le d$. Moreover, for $0\le i\le d$ we have
\begin{gather*}
e_0^{+}U_i\subseteq U_{i+1},\qquad e_1^{-}U_i\subseteq U_{i+1},\\
e_0^{-}U_i\subseteq U_{i-1},\qquad e_1^{+}U_i\subseteq U_{i-1}.
\end{gather*}
\end{lem}
\begin{1def}
\label{defd}
Let $V$ denote a finite-dimensional irreducible $U_q(L(\mathfrak{sl}_2))$-module. By the {\it diameter} of $V$ we mean the scalar $d$ from Lemma {\ref {lem:dec}}. By the {\it type} of $V$ we mean the scalar $\gamma$ from Lemma {\ref {lem:dec}}. The sequence $\{U_i\}_{i=0}^{d}$ is the {\it weight space} decomposition of $V$ relative to $K_0$ and $K_1$.
\end{1def}

\begin{lem}
\label{lem:415}
Let $V$ denote a finite-dimensional irreducible $U_q(L(\mathfrak{sl}_2))$-module of type $1$. Let $\{U_i\}_{i=0}^{d}$ denote the weight space decomposition of $V$ from Lemma {\ref {lem:dec}}. Then $U_0$ and $U_d$ have dimension $1$.
\end{lem}
\begin{proof}
By \cite[\rm Lemma~3.12]{non}.
\end{proof}

We now give a detailed description of the finite-dimensional irreducible $U_q(L(\mathfrak{sl}_2))$-modules of type $1$ and diameter $1$.
\begin{lem}
\label{eg1}
\rm
(See \cite[\rm Section~4]{cp}.)
\it
There exists a family of finite-dimensional irreducible $U_q(L(\mathfrak{sl}_2))$-modules
\begin{equation}
\label{evae}
V(1,a) \qquad \qquad 0\ne a\in \mathbb{F}
\end{equation}
with the following property: $V(1,a)$ has a basis $v_0,v_1$ with respect to which the Chevalley generators are represented by the following matrices:
\begin{gather*}
e_0^+:\,
\begin{bmatrix}
    0       & 0 \\
    a       & 0
\end{bmatrix},\qquad
e_0^-:\,
\begin{bmatrix}
    0       & a^{-1} \\
    0       & 0
\end{bmatrix},\qquad
K_0:\,
\begin{bmatrix}
    q^{-1}       & 0 \\
    0       & q
\end{bmatrix},\\
e_1^+:\,
\begin{bmatrix}
    0       & 1 \\
    0       & 0
\end{bmatrix},\qquad
e_1^-:\,
\begin{bmatrix}
    0       & 0 \\
    1       & 0
\end{bmatrix},\qquad
K_1:\,
\begin{bmatrix}
    q       & 0 \\
    0       & q^{-1}
\end{bmatrix}.
\end{gather*}
Moreover $V(1,a)$ has type $1$ and diameter $1$. Every finite-dimensional irreducible $U_q(L(\mathfrak{sl}_2))$-module of type $1$ and diameter $1$ is isomorphic to exactly one of the modules (\ref {evae}).
\end{lem}

\begin{1def}
\label{def:eva}
Referring to Lemma {\ref {eg1}}, we call $a$ the {\it evaluation parameter} of $V(1,a)$.
\end{1def}

We now describe how the $U_q(L(\mathfrak{sl}_2))$-modules (\ref {evae}) look from the equitable point of view.
\begin{lem}
\label{lem:qeva}
Recall the basis $v_0,v_1$ of $V(1,a)$ from Lemma {\ref {eg1}}. With respect to this basis the equitable generators are represented by the following matrices:
\begin{gather*}
X_{01}:\,
\begin{bmatrix}
    q^{-1}       & (q-q^{-1})a^{-1} \\
    0       & q
\end{bmatrix},\qquad
X_{23}:\,
\begin{bmatrix}
    q       & 0 \\
    q-q^{-1}       & q^{-1}
\end{bmatrix},\qquad \\
X_{12}:\,
\begin{bmatrix}
    q       & q^{-1}-q \\
    0       & q^{-1}
\end{bmatrix},\qquad
X_{30}:\,
\begin{bmatrix}
    q^{-1}       & 0 \\
    (q^{-1}-q)a      & q
\end{bmatrix},\qquad \quad\\
X_{13}:\,
\begin{bmatrix}
    q       & 0 \\
    0       & q^{-1}
\end{bmatrix},\qquad \qquad
X_{31}:\,
\begin{bmatrix}
    q^{-1}       & 0 \\
    0       & q
\end{bmatrix}.
\end{gather*}
\end{lem}
\begin{proof}
By Lemma {\ref {lem:lp2}}.
\end{proof}

\begin{lem}
\label{lem:tra}
For $0\ne a\in \mathbb{F}$, the following equations hold on the $U_q(L(\mathfrak{sl}_2))$-module $V(1,a)$ from Lemma {\ref {eg1}}:
\begin{gather}
{\rm{tr}}(X_{01}X_{23})=2+(q-q^{-1})^2a^{-1},\label{tr1}\\
{\rm{tr}}(X_{12}X_{30})=2+(q-q^{-1})^2a, \label{tr2}
\end{gather}
where ${\rm{tr}}$ denotes trace.
\end{lem}
\begin{proof}
Use Lemma {\ref {lem:qeva}}.
\end{proof}

\begin{lem}
For $0\ne a\in \mathbb{F}$, let $v_0,v_1$ denote the basis of $V(1,a)$ from Lemma {\ref {eg1}}. Let $U_0,U_1$ denote the weight space decomposition of $V(1,a)$. Then $U_i$ is spanned by $v_i$ for $i=0,1$.
\end{lem}
\begin{proof}
Use the matrices that represent $K_0$ and $K_1$ in Lemma {\ref {eg1}}.
\end{proof}

\section{The Drinfel'd polynomial of a $U_q(L(\mathfrak{sl}_2))$-module}
In this section we recall the Drinfel'd polynomial of a finite-dimensional irreducible $U_q(L(\mathfrak{sl}_2))$-module of type $1$.

\begin{1def}
\label{def:51}
Let $V$ denote a finite-dimensional irreducible $U_q(L(\mathfrak{sl}_2))$-module of type $1$ and diameter $d$. Let $\{U_i\}_{i=0}^{d}$ denote the weight space decomposition of $V$ from Lemma {\ref {lem:dec}}. Pick $j\in\mathbb{N}$. By Lemmas {\ref {lem:dec}} and {\ref {lem:415}}, we see that $U_0$ is an eigenspace for $(e_1^+)^j(e_0^+)^j$; let $\sigma_j$ denote the corresponding eigenvalue. Observe that $\sigma_0=1$, and $\sigma_j=0$ if $j>d$.
\end{1def}

\begin{1def}
(See \cite[\rm Section~3.5]{cp}.)
\label{def:52}
Let $V$ denote a finite-dimensional irreducible $U_q(L(\mathfrak{sl}_2))$-module of type $1$ and diameter $d$. We define a polynomial $P_V \in\mathbb{F}[z]$ by
\begin{gather*}
P_V=\sum_{i=0}^{d} \frac{(-1)^i\sigma_iz^i}{([i]_q^{!})^2},
\end{gather*}
where the scalars $\sigma_i$ are from Definition {\ref {def:51}}. Observe that $P_V$ has constant coefficient $1$. We call $P_V$ the {\it Drinfel'd polynomial} of $V$.
\end{1def}

The polynomial $P_V$ has the following property.

\begin{thm}
\label{thm:57}
\rm
(See \cite[\rm p.~261]{cp}.)
\it
The map $V\mapsto P_V$ induces a bijection between the following two sets:
\begin{enumerate}
\item[\rm(i)] the isomorphism classes of finite-dimensional irreducible $U_q(L(\mathfrak{sl}_2))$-modules of type $1$;
\item[\rm(ii)] the polynomials in $\mathbb{F}[z]$ that have constant coefficient $1$.
\end{enumerate}
\end{thm}

We have been discussing the Drinfel'd polynomial $P_V$. We now mention a related polynomial $Q_V$.
\begin{1def}
\label{def:51b}
Let $V$ denote a finite-dimensional irreducible $U_q(L(\mathfrak{sl}_2))$-module of type $1$ and diameter $d$. Let $\{U_i\}_{i=0}^{d}$ denote the weight space decomposition of $V$ from Lemma {\ref {lem:dec}}. Pick $j\in\mathbb{N}$. By Lemmas {\ref {lem:dec}} and {\ref {lem:415}}, we see that $U_0$ is an eigenspace for $(e_0^-)^j(e_1^-)^j$; let $\mu_j$ denote the corresponding eigenvalue. Observe $\mu_0=1$, and $\mu_j=0$ if $j>d$.
\end{1def}

\begin{1def}
(See \cite[\rm Section~3.5]{cp}.)
\label{def:52b}
Let $V$ denote a finite-dimensional irreducible $U_q(L(\mathfrak{sl}_2))$-module of type $1$ and diameter $d$. We define a polynomials $ Q_V\in\mathbb{F}[z]$ by
\begin{gather*}
Q_V=\sum_{i=0}^{d} \frac{(-1)^i\mu_iz^i}{([i]_q^{!})^2},
\end{gather*}
where the scalars $\mu_i$ are from Definition {\ref {def:51b}}. Observe that $Q_V$ have constant coefficient $1$.
\end{1def}

\begin{lem}
\label{lem:54}
\rm
(See \cite[\rm p.~269]{cp}.)
\it
The polynomials $P_V$ and $Q_V$ are partners in the sense of Definition {\ref {def:part}}.
\end{lem}

\begin{note}
The polynomials $P, Q$ in \cite[\rm Section~3.5]{cp} are related to our polynomials $P_V, Q_V$ as follows: $P_V(z)=P(q^{-1}z)/P(0)$ and $Q_V(z)=Q(q^{-1}z)/Q(0)$. This is explained in \cite[\rm p.~268, 269]{cp}.
\end{note}

We now compute $P_V$ and $Q_V$ for the case in which $V$ has type $1$ and diameter $1$.
\begin{lem}
\label{eg2}
Pick $0\ne a\in\mathbb{F}$. Let $V$ denote the $U_q(L(\mathfrak{sl}_2))$-module $V(1,a)$ from Lemma {\ref {eg1}}. Then $P_V(z)=1-az$ and $Q_V(z)=1-a^{-1}z$.
\end{lem}
\begin{proof}
Using the matrices that represent $e_0^+,e_1^+$ in Lemma {\ref {eg1}} along with Definition {\ref {def:51}}, we obtain $\sigma_1=a$. Therefore $P_V(z)=1-az$ in view of Definition {\ref {def:52}}. Similarly using Lemma {\ref {eg1}} and Definition {\ref {def:51b}}, we obtain $\mu_1=a^{-1}$. Therefore $Q_V(z)=1-a^{-1}z$ in view of Definition {\ref {def:52b}}.
\end{proof}

\section{The positive part of $U_q(L(\mathfrak{sl}_2))$}
In this section we recall a certain algebra $U_q^+$ called the positive part of $U_q(L(\mathfrak{sl}_2))$.
\begin{1def}
\label{def:61}
(See \cite[\rm Definition~1.1]{non}.)
Let $U_q^+$ denote the $\mathbb{F}$-algebra with generators $x,y$ and relations
\begin{gather*}
x^3y-[3]_qx^2yx+[3]_qxyx^2-yx^3=0,\\
y^3x-[3]_qy^2xy+[3]_qyxy^2-xy^3=0.
\end{gather*}
We call $x,y$ the {\it standard generators} for $U_q^+$.
\end{1def}

\begin{rmk}
\label{rmk:6}
(See \cite[\rm Corollary~3.2.6]{luz}.)
There exists an injective $\mathbb{F}$-algebra homomorphism from $U_q^+$ to $U_q(L(\mathfrak{sl}_2))$ that sends $x\mapsto e_0^+$ and $y\mapsto e_1^+$. Consequently we call $U_q^+$ the {\it positive part} of $U_q(L(\mathfrak{sl}_2))$.
\end{rmk}

Comparing the relations in Definition {\ref {def:box}} with the relations in Definition {\ref {def:61}}, we obtain an $\mathbb{F}$-algebra homomorphism $\kappa: U_q^+ \to \square_q$ that sends $x\mapsto x_{0}$ and $y\mapsto x_{2}$.
\begin{lem}
\label{lem:2}
\rm
(See \cite[Proposition~5.5]{p1}.)
\it
The above homomorphism $\kappa$ is injective.
\end{lem}

Recall the map $\psi$ from above Lemma {\ref {lem:1}}. Consider the composition
\begin{equation}
\label{m1}
\psi\circ\kappa:
\begin{tikzcd}
 U_q^+\arrow{r}{\kappa} & \square_q \arrow{r}{\psi}& U_q(L(\mathfrak{sl}_2)).
\end{tikzcd}
\end{equation}
Note that the map $\psi\circ\kappa$ is different from the map in Remark {\ref {rmk:6}}. Let $V$ denote a finite-dimensional irreducible $U_q(L(\mathfrak{sl}_2))$-module of type $1$. Consider the pullback of $V$ via $\psi\circ\kappa$. We now describe the resulting $U_q^+$-module $V$.
\begin{lem}
\label{irre}
\rm
(See \cite[\rm Proposition~12.1]{non}.)
\it
The above $U_q^+$-module $V$ is irreducible if and only if $P_V(1)\ne 0$.
\end{lem}

\begin{1def}
(See \cite[\rm Definition~1.3]{non}.)
Let $V$ denote a finite-dimensional $U_q^+$-module. This module is called {\it NonNil} whenever the standard generators $x,y$ are not nilpotent on $V$.
\end{1def}

\begin{lem}
\label{lem:64}
\rm
(See \cite[\rm Corollary~2.8]{non}.)
\it
Let $V$ denote a NonNil finite-dimensional irreducible $U_q^+$-module. Then the standard generators $x,y$ are semisimple on $V$. Moreover there exist $d\in\mathbb{N}$ and nonzero scalars $\gamma,\gamma^\prime\in\mathbb{F}$ such that the set of distinct eigenvalues of $x$ (resp. $y$) on $V$ is $\{\gamma q^{d-2i}\}_{i=0}^d$ (resp. $\{\gamma^\prime q^{d-2i}\}_{i=0}^d$).
\end{lem}

\begin{1def}
(See \cite[\rm Definition~2.9]{non}.)
Let $V$ denote a NonNil finite-dimensional irreducible $U_q^+$-module. By the {\it diameter} of $V$ we mean the scalar $d$ from Lemma {\ref {lem:64}}. By the {\it type} of $V$ we mean the ordered pair $(\gamma,\gamma^\prime)$ from Lemma {\ref {lem:64}}.
\end{1def}

We will use the following notation.
\begin{1def}
\label{def:not}
Let $\mathcal{A},\mathcal{B}$ denote $\mathbb{F}$-algebras and $\varphi: \mathcal{A}\to \mathcal{B}$ denote an $\mathbb{F}$-algebra homomorphism. Let $\varphi^\sharp$ denote the map that sends each $\mathcal{B}$-module to its pullback via $\varphi$.
\end{1def}

Referring to ({\ref{m1}}) and Definition {\ref {def:not}}, note that
\begin{equation}
\label{k}
(\psi\circ\kappa)^\sharp=\kappa^\sharp\circ\psi^\sharp.
\end{equation}
In the next result we describe $(\psi\circ\kappa)^\sharp$. In Section 8 we describe the maps $\kappa^\sharp, \psi^\sharp$.
\begin{thm}
\label{thm:68}
\rm
(See \cite[\rm Theorems~1.6, 1.7]{non}.)
\it
The map $(\psi\circ\kappa)^\sharp$ gives a bijection between the following two sets:
\begin{enumerate}
\item[\rm(i)] the isomorphism classes of finite-dimensional irreducible $U_q(L(\mathfrak{sl}_2))$-modules $V$ of type $1$ such that $P_V(1)\ne 0$;
\item[\rm(ii)] the isomorphism classes of NonNil finite-dimensional irreducible $U_q^+$-modules of type $(1,1)$.
\end{enumerate}
\end{thm}

Let $V$ denote a NonNil finite-dimensional irreducible $U_q^+$-module of type $(1,1)$. Via Theorem {\ref{thm:68}}, the vector space $V$ becomes a finite-dimensional irreducible $U_q(L(\mathfrak{sl}_2))$-module of type $1$ such that $P_V(1)\ne 0$.
\begin{1def}
\label{driuq}
Let $V$ denote a NonNil finite-dimensional irreducible $U_q^+$-module of type $(1,1)$. Let $P_V$ (resp. $Q_V$) denote the polynomial from Definition {\ref {def:52}} (resp. Definition {\ref {def:52b}}) associated with the $U_q(L(\mathfrak{sl}_2))$-module $V$ from Theorem {\ref{thm:68}}.
\end{1def}

We now give a detailed description of the NonNil finite-dimensional irreducible $U_q^+$-modules of type $(1,1)$ and diameter $1$. Pick $0\ne a\in\mathbb{F}$. Via Theorem {\ref {thm:68}}, the $U_q(L(\mathfrak{sl}_2))$-module $V(1,a)$ from Lemma {\ref {eg1}} becomes a NonNil $U_q^+$-module.
\begin{lem}
\label{eg3}
The NonNil $U_q^+$-modules
\begin{equation}
\label{evae1}
V(1,a) \qquad \qquad a\in \mathbb{F}\backslash \{0,1\}
\end{equation}
are irreducible. With respect to the basis $v_0,v_1$ from Lemma {\ref {eg1}}, the standard generators $x,y$ are represented by the following matrices:
\begin{gather}
\label{matr1}
x:\,
\begin{bmatrix}
    q^{-1}       & (q-q^{-1})a^{-1} \\
    0       & q
\end{bmatrix},\qquad \quad
y:\,
\begin{bmatrix}
    q       & 0 \\
    q-q^{-1}       & q^{-1}
\end{bmatrix}.
\end{gather}
Moreover $V(1,a)$ has type $(1,1)$ and diameter $1$. Every NonNil finite-dimensional irreducible $U_q^+$-module of type $(1,1)$ and diameter $1$ is isomorphic to exactly one of the modules (\ref {evae1}).
\end{lem}
\begin{proof}
Pick $0\ne a\in\mathbb{F}$. By Lemma {\ref {eg2}} and Lemma {\ref {irre}}, the $U_q^+$-module $V(1,a)$ is irreducible if and only if $a\ne 1$. The matrices in ({\ref {matr1}}) are obtained from Lemma {\ref {lem:qeva}} and Theorem {\ref{thm:68}}. The remaining assertions follow by Lemma {\ref {eg1}} and Theorem {\ref{thm:68}}.
\end{proof}

\begin{lem}
Pick $a\in \mathbb{F}\backslash \{0,1\}$. Let $V$ denote the $U_q^+$-module $V(1,a)$ from Lemma {\ref {eg3}}. Then $P_V(z)=1-az$ and $Q_V(z)=1-a^{-1}z$.
\end{lem}
\begin{proof}
By Lemma {\ref {eg2}} and Definition {\ref {driuq}}.
\end{proof}

\section{The type of a finite-dimensional irreducible $\square_q$-module}
In Theorem {\ref {thm:68}}, we described the map $(\psi\circ\kappa)^\sharp$. We will describe the maps $\kappa^\sharp, \psi^\sharp$ in Section 9. The main result for $\kappa^\sharp$ (resp. $\psi^\sharp$) is Theorem {\ref {rmk:10.4}} (resp. Theorem {\ref {thm:101}}). In Sections 7, 8 we obtain the results needed to prove the above theorems.

Let $V$ denote a finite-dimensional irreducible $\square_q$-module. Our next goal is to show that each generator $x_i$ of $\square_q$ is semisimple on $V$, and we find its eigenvalues.

We begin with some results concerning the relation ({\ref {equ:11}}).

\begin{lem}
\label{lem:82}
\rm
(See \cite[\rm Lemma~11.2]{qtet}.)
\it
Let $V$ denote a vector space over $\mathbb{F}$ with finite positive dimension. Let $C,D\in {\rm End}(V)$. Then for all nonzero $\theta\in \mathbb{F}$ the following are equivalent:
\begin{enumerate}
\item[\rm(i)] The expression $qCD-q^{-1}DC-(q-q^{-1})I$ vanishes on $V_C(\theta)$;
\item[\rm(ii)] $(D-\theta^{-1}I)V_C(\theta)\subseteq V_C(q^{-2}\theta)$.
\end{enumerate}
\end{lem}

\begin{lem}
\label{lem:83}
\rm
(See \cite[\rm Lemma~11.3]{qtet}.)
\it
Let $V$ denote a vector space over $\mathbb{F}$ with finite positive dimension. Let $C,D\in {\rm End}(V)$. Then for all nonzero $\theta\in \mathbb{F}$ the following are equivalent:
\begin{enumerate}
\item[\rm(i)] The expression $qCD-q^{-1}DC-(q-q^{-1})I$ vanishes on $V_D(\theta)$;
\item[\rm(ii)] $(C-\theta^{-1}I)V_D(\theta)\subseteq V_D(q^{2}\theta)$.
\end{enumerate}
\end{lem}

\begin{lem}
\label{lem:84}
\rm
(See \cite[\rm Lemma~11.4]{qtet}.)
\it
Let $V$ denote a vector space over $\mathbb{F}$ with finite positive dimension. Let $C,D\in {\rm End}(V)$ such that
\begin{gather*}
\frac{qCD-q^{-1}DC}{q-q^{-1}}=I.
\end{gather*}
Then for all nonzero $\theta\in \mathbb{F}$,
\begin{gather*}
\sum_{n=0}^{\infty}V_C(q^{-2n}\theta)=\sum_{n=0}^{\infty}V_D(q^{2n}\theta^{-1}).
\end{gather*}
\end{lem}

\begin{lem}
\label{lem:85}
\rm
(See \cite[\rm Lemma~11.5]{qtet}.)
\it
With the notation and assumptions of Lemma {\ref {lem:84}},
\begin{gather*}
\dim V_C(\theta)=\dim V_D(\theta^{-1}).
\end{gather*}
\end{lem}

\begin{lem}
\label{lem:86}
Let $V$ denote a finite-dimensional irreducible $\square_q$-module. Then for all nonzero $\theta\in \mathbb{F}$ we have
\begin{gather*}
\dim V_{x_0}(\theta)=\dim V_{x_1}(\theta^{-1})=\dim V_{x_2}(\theta)=\dim V_{x_3}(\theta^{-1}).
\end{gather*}
\end{lem}
\begin{proof}
By (\ref {equ:11}) we have
\begin{gather}
x_0\to x_1\to x_2\to x_3 \label{equ:81}
\end{gather}
where $r\to s$ means
\begin{gather*}
\frac{qrs-q^{-1}sr}{q-q^{-1}}=1.
\end{gather*}
Applying Lemma {\ref {lem:85}} to each arrow in (\ref {equ:81}) we obtain the result.
\end{proof}

We mention a result concerning the relation ({\ref {equ:12}}).
\begin{lem}
\label{lem:81}
\rm
(See \cite[\rm Lemma~11.1]{qtet}.)
\it
Let $V$ denote a finite-dimensional vector space over $\mathbb{F}$. Let $C,D\in {\rm End}(V)$. Then for all nonzero $\theta\in \mathbb{F}$ the following are equivalent:
\begin{enumerate}
\item[\rm(i)] The expression $C^3D-[3]_qC^2DC+[3]_qCDC^2-DC^3$ vanishes on $V_C(\theta)$;
\item[\rm(ii)] $DV_C(\theta)\subseteq V_C(q^2\theta)+V_C(\theta)+V_C(q^{-2}\theta)$.
\end{enumerate}
\end{lem}

\begin{lem}
\label{lem:88}
Let $V$ denote a finite-dimensional irreducible $\square_q$-module. Then each generator $x_i$ of $\square_q$ is semisimple on $V$. Moreover there exist $d\in \mathbb{N}$ and nonzero $\gamma\in\mathbb{F}$ with the following property:
\begin{enumerate}
\item[\rm(i)] for $x_0$ and $x_2$, the set of distinct eigenvalues on $V$ is $\{\gamma q^{d-2i}\}_{i=0}^{d}$;
\item[\rm(ii)] for $x_1$ and $x_3$, the set of distinct eigenvalues on $V$ is $\{\gamma^{-1}q^{d-2i}\}_{i=0}^{d}$.
\end{enumerate}
\end{lem}
\begin{proof}
Since $\mathbb{F}$ is algebraically closed and $V$ has finite positive dimension, there exists $c\in\mathbb{F}$ such that $V_{x_0}(c)\ne 0$. By \cite[\rm Proposition~5.2]{y1}, the action of $x_0$ is invertible on $V$. Therefore $0$ is not an eigenvalue of $x_0$ on $V$, so $c\ne 0$. By this and since $q$ is not a root of $1$, the scalars $c,q^2c,q^4c,\dots$ are mutually distinct. Consequently these scalars can not all be eigenvalues for $x_0$ on $V$. Therefore there exists a nonzero $\tau\in\mathbb{F}$ such that $V_{x_0}(\tau)\ne 0$ and $V_{x_0}(q^2\tau)= 0$. Moreover there exists $d\in\mathbb{N}$ such that $V_{x_0}(q^{-2n}\tau)$ is nonzero for $0\le n\le d$ and zero for $n=d+1$. We will show that
\begin{equation}
V_{x_0}(\tau)+V_{x_0}(q^{-2}\tau)+\dots+V_{x_0}(q^{-2d}\tau) \label{equ:82}
\end{equation}
is equal to $V$. Our strategy is to show that the subspace (\ref {equ:82}) is a $\square_q$-submodule of $V$. By construction (\ref {equ:82}) is invariant under $x_0$. By Lemma {\ref {lem:82}} and since $V_{x_0}(q^{-2d-2}\tau)=0$ the subspace (\ref {equ:82}) is invariant under $x_1$. By Lemma {\ref {lem:83}} and since $V_{x_0}(q^{2}\tau)=0$ the subspace (\ref {equ:82}) is invariant under $x_3$. By Lemma {\ref {lem:81}}, and since each of $V_{x_0}(q^{2}\tau), V_{x_0}(q^{-2d-2}\tau)$ is zero, the subspace (\ref {equ:82}) is invariant under $x_2$. We have now shown that (\ref {equ:82}) is invariant under each generator $x_i$ of $\square_q$, so (\ref {equ:82}) is a $\square_q$-submodule of $V$. Each term in (\ref {equ:82}) is nonzero and there is at least one term so (\ref {equ:82}) is nonzero. By these comments and since the $\square_q$-module $V$ is irreducible we find (\ref {equ:82}) is equal to $V$. Let $\gamma=\tau q^{-d}$. Observe that $\gamma\ne0$. By the above comments, the action of $x_0$ on $V$ is semisimple with eigenvalues $\{\gamma q^{d-2n}\}_{n=0}^d$. The remaining assertions follow in view of Lemma {\ref {lem:86}}.
\end{proof}

\begin{1def}
Let $V$ denote a finite-dimensional irreducible $\square_q$-module. By the {\it diameter} of $V$ we mean the scalar $d$ from Lemma {\ref {lem:88}}. By the {\it type} of $V$ we mean the scalar $\gamma$ from Lemma {\ref {lem:88}}.
\end{1def}

\begin{note}
By the comments below Lemma {\ref {lem:gam}}, we focus on the finite-dimensional irreducible $\square_q$-modules of type $1$.
\end{note}

\begin{1def}
\label{def:812}
Let $V$ denote a finite-dimensional irreducible $\square_q$-module of type $1$ and diameter $d$. For each generator $x_i$ of $\square_q$ we define a decomposition of $V$ which we call $[i,i+1]$. The decomposition $[i,i+1]$ has diameter $d$. For $0\le n\le d$ the $n$th component of $[i,i+1]$ is the eigenspace of $x_i$ on $V$ associated with the eigenvalue $q^{d-2n}$. Let $[i+1,i]$ denote the inversion of the decomposition $[i,i+1]$.
\end{1def}

\begin{lem}
\label{lem:813}
Let $V$ denote a finite-dimensional irreducible $\square_q$-module of type $1$. Then the shapes of the following decompositions of $V$ coincide:
\begin{gather*}
[0,1],\qquad [2,1],\qquad [2,3],\qquad [0,3].
\end{gather*}
\end{lem}
\begin{proof}
By Lemma {\ref {lem:86}}.
\end{proof}

\section{Flags on finite-dimensional irreducible $\square_q$-modules}
Let $V$ denote a finite-dimensional irreducible $\square_q$-module of type $1$. In this section we discuss some flags on $V$. We begin our discussion by recalling the notion of a flag.

Let $V$ denote a vector space over $\mathbb{F}$ with finite positive dimension. Let $\{s_i\}_{i=0}^d$ denote a sequence of positive integers whose sum is the dimension of $V$. By a {\it flag} on $V$ of {\it shape} $\{s_i\}_{i=0}^d$ we mean a nested sequence $U_0\subseteq U_1\subseteq \dots \subseteq U_d$ of subspaces of $V$ such that the dimension of $U_i$ is $s_0+s_1+\dots+s_i$ for $0\le i\le d$. Observe that $U_d=V$. We call $U_i$ the $i$th {\it component} of the flag. We call $d$ the {\it diameter} of the flag.

The following construction yields a flag on $V$. Let $\{V_i\}_{i=0}^d$ denote a decomposition of $V$ with shape $\{s_i\}_{i=0}^d$. Define
\begin{gather*}
U_i=V_0+V_1+\dots + V_i \qquad (0\le i\le d).
\end{gather*}
Then the sequence $U_0\subseteq U_1\subseteq \dots \subseteq U_d$ is a flag on $V$ of shape $\{s_i\}_{i=0}^d$. We say this flag is {\it induced} by the decomposition $\{V_i\}_{i=0}^d$.

We now recall what it means for two flags to be {\it opposite}. Suppose we are given two flags on $V$ with the same diameter: $U_0\subseteq U_1\subseteq \dots \subseteq U_d$ and $U_0^\prime\subseteq U_1^\prime\subseteq \dots \subseteq U_d^\prime$. We say these flags are {\it opposite} whenever there exists a decomposition $\{V_i\}_{i=0}^d$ of $V$ that induces $U_0\subseteq U_1\subseteq \dots \subseteq U_d$ and its inversion $\{V_{d-i}\}_{i=0}^d$ induces $U_0^\prime\subseteq U_1^\prime\subseteq \dots \subseteq U_d^\prime$. In this case
\begin{gather*}
U_i\cap U_j^\prime=0\qquad (0\le i,j\le d,\quad i+j<d)
\end{gather*}
and
\begin{gather}
V_n=U_n\cap U_{d-n}^\prime \qquad (0\le n\le d). \label{equ:92}
\end{gather}
The decomposition $\{V_i\}_{i=0}^d$ is uniquely determined by the given flags.

We now bring in the finite-dimensional irreducible $\square_q$-modules.

\begin{thm}
\label{thm:91}
Let $V$ denote a finite-dimensional irreducible $\square_q$-module of type $1$. Then there exists a collection of flags on $V$, denoted $[h],\, h\in\mathbb{Z}_4$, that have the following property: for each generator $x_i$ of $\square_q$ the decomposition $[i,i+1]$ of $V$ induces $[i]$ and its inversion $[i+1,i]$ induces $[i+1]$.
\end{thm}
\begin{proof}
For all $h\in\mathbb{Z}_4$ let $[h]$ denote the flag on $V$ induced by the decomposition $[h,h-1]$. By Lemma {\ref {lem:84}} (with $C=x_{h-1}$ and $D=x_h$) the flag on $V$ induced by $[h,h+1]$ is equal to $[h]$. The result follows.
\end{proof}

\begin{thm}
\label{thm:92}
Let $V$ denote a finite-dimensional irreducible $\square_q$-module of type $1$. Then for $i\in\mathbb{Z}_4$ the flags $[i],[i+1]$ are opposite.
\end{thm}
\begin{proof}
We invoke Theorem {\ref {thm:91}}. The flags $[i],[i+1]$ are opposite since the decomposition $[i,i+1]$ induces $[i]$ and its inversion $[i+1,i]$ induces $[i+1]$.
\end{proof}

\begin{note}
Theorem {\ref {thm:92}} can be strengthened as follows: the flags $[i],\, i\in\mathbb{Z}_4$ are mutually opposite. This follows from Theorem {\ref {thm:eta}} and \cite[\rm Theorem~16.3]{qtet}. We don't discuss the details since we do not need this fact.
\end{note}

\begin{thm}
\label{thm:93}
Let $V$ denote a finite-dimensional irreducible $\square_q$-module of type $1$ and diameter $d$. Pick a generator $x_i$ of $\square_q$ and consider the corresponding decomposition $[i,i+1]$ of $V$ from Definition {\ref {def:812}}. For $0\le n\le d$ the $n$th component of $[i,i+1]$ is the intersection of the following two sets:
\begin{enumerate}
\item[\rm(i)] component $n$ of the flag $[i]$;
\item[\rm(ii)] component $d-n$ of the flag $[i+1]$;.
\end{enumerate}
\end{thm}
\begin{proof}
By Theorem {\ref {thm:91}} and (\ref {equ:92}).
\end{proof}

\section{The maps $\kappa^\sharp, \psi^\sharp$}
In Theorem {\ref {thm:68}} we described the map $(\psi\circ\kappa)^\sharp$. In this section we describe the maps $\kappa^\sharp, \psi^\sharp$. We first show that $\kappa^\sharp$ gives a bijection between the isomorphism classes of finite-dimensional irreducible $\square_q$-modules of type $1$ and the isomorphism classes of NonNil finite-dimensional irreducible $U_q^+$-modules of type $(1,1)$. Then we show that the map $\psi^\sharp$ gives a bijection between the isomorphism classes of finite-dimensional irreducible $U_q(L(\mathfrak{sl}_2))$-modules $V$ of type $1$ such that $P_V(1)\ne 0$ and the isomorphism classes of finite-dimensional irreducible $\square_q$-modules of type $1$. After establishing these results, we define the Drinfel'd polynomial of a finite-dimensional irreducible $\square_q$-module of type $1$ and prove Proposition \ref{thm:36}.

\begin{pro}
\label{pro:10.1}
Let $V$ denote a finite-dimensional irreducible $\square_q$-module of type $1$. Let $W$ denote a nonzero subspace of $V$ such that $x_0W\subseteq W$ and $x_2W\subseteq W$. Then $W=V$.
\end{pro}
\begin{proof}
Without loss we may assume that $W$ is irreducible as a module for $x_0,x_2$. Let
$\{V_i\}_{i=0}^d$ denote the decomposition $[0,1]$ of $V$ and let $\{V_i^\prime\}_{i=0}^d$ denote the the decomposition $[2,3]$ of $V$. Recall that $x_0$ (resp. $x_2$) is semisimple on $V$ with eigenspaces $\{V_i\}_{i=0}^d$ (resp. $\{V_i^\prime\}_{i=0}^d$). By this and since $W$ is invariant under each of $x_0,x_2$ we find
\begin{gather}
W=\sum_{n=0}^d W\cap V_n, \qquad\qquad W=\sum_{n=0}^d W\cap V_n^\prime .\label{equ:10.1}
\end{gather}
Define
\begin{gather}
m=\min\{n\mid 0\le n\le d,\quad W\cap V_n\ne 0\},\label{equ:10.2}\\
m^\prime=\min\{n\mid 0\le n\le d,\quad W\cap V_n^\prime \ne 0\}.\nonumber
\end{gather}
We claim that $m=m^\prime$. To prove this claim, we first show that $m\le m^\prime$. Suppose $m>m^\prime$. By (\ref {equ:10.2}) and the equation on the left in (\ref {equ:10.1}), the space $W$ is contained in component $d-m$ of the flag $[1]$. By construction $W$ has nonzero intersection with component $m^\prime$ of the flag $[2]$. By Theorem {\ref {thm:92}} (with $i=1$) and $m>m^\prime$, the component $d-m$ of $[1]$ has zero intersection with component $m^\prime$ of $[2]$, which is a contradiction. Therefore $m\le m^\prime$. Applying the above argument to $^{\rho^2}V$ we obtain $m^\prime\le m$. We have shown $m\le m^\prime\le m$. Therefore the claim is proved. The claim implies that the component $d-m$ of the flag $[1]$ contains $W$, the component $d-m$ of the flag $[3]$ contains $W$, the component $m$ of the flag $[0]$ has nonzero intersection with $W$, and the component $m$ of the flag $[2]$ has nonzero intersection with $W$. We can now easily show $W=V$. Since the $\square_q$-module $V$ is irreducible, it suffices to show that $W$ is invariant under each generator $x_i$ of $\square_q$. By construction $W$ is invariant under $x_0$ and $x_2$. Let $x_r$ denote one of $x_1,x_3$. We show $x_rW\subseteq W$. Let $W^\prime$ denote the span of the set of vectors in $W$ that are eigenvectors for $x_r$. By construction $W^\prime\subseteq W$ and $x_rW^\prime\subseteq W^\prime$. We show $W^\prime=W$. To this end we show that $W^\prime$ is nonzero and invariant under each of $x_0,x_2$. We now show $W^\prime\ne 0$. By Theorem {\ref {thm:93}} the intersection of the component $d-m$ of the flag $[1]$ and the component $m$ of the flag $[2]$ is equal to component $d-m$ of the decomposition $[1,2]$ which is an eigenspace for $x_1$. By this and the comment after the preliminary claim we have $W^\prime\ne 0$ for $x_r=x_1$. Similarly we show $W^\prime\ne 0$ for $x_r=x_3$. We now show $x_0W^\prime\subseteq W^\prime$. To this end we pick $v\in W^\prime$ and show $x_0v\in W^\prime$. Without loss we may assume that $v$ is an eigenvector for $x_r$; let $\theta$ denote the corresponding eigenvalue. By Lemma {\ref {lem:88}} we have $\theta\ne 0$. Recall $v\in W^\prime$ and $W^\prime\subseteq W$ so $v\in W$. The space $W$ is $x_0$-invariant so $x_0v\in W$. By these comments $(x_0-\theta^{-1}I)v\in W$. By Lemma {\ref {lem:82}} or Lemma {\ref {lem:83}} the vector $(x_0-\theta^{-1}I)v$ is contained in an eigenspace of $x_r$, so $(x_0-\theta^{-1}I)v\in W^\prime$. By this and since $v\in W^\prime$ we have $x_0v\in W^\prime$. We have shown $x_0W^\prime\subseteq W^\prime$. Similarly we show $x_2W^\prime\subseteq W^\prime$. So far we have shown that $W^\prime$ is nonzero and invariant under each of $x_0,x_2$. By the irreducibility of $W$ we have $W^\prime=W$, so $x_rW\subseteq W$. It follows that $W$ is $\square_q$-invariant. The space $V$ is irreducible as a $\square_q$-module so $W=V$.
\end{proof}
Recall the map $\kappa$ from Lemma {\ref {lem:2}}. Let $V$ denote a finite-dimensional irreducible $\square_q$-module of type $1$. Pulling back the $\square_q$-module structure on $V$ via $\kappa$, we turn $V$ into a $U_q^+$-module.
\begin{lem}
\label{thm:10.2}
The above $U_q^+$-module $V$ is NonNil and irreducible of type $(1,1)$.
\end{lem}
\begin{proof}
The $U_q^+$-module $V$ is irreducible by Proposition {\ref {pro:10.1}}. By Lemma {\ref {lem:88}} and the construction, for each generator $x,y$ of $U_q^+$ the action on $V$ is semisimple with eigenvalues $\{q^{d-2n}\mid 0\le n\le d\}$. Therefore the $U_q^+$-module $V$ is NonNil and type $(1,1)$. The result follows.
\end{proof}

\begin{lem}
\label{thm:10.3}
Let $V$ denote a NonNil finite-dimensional irreducible $U_q^+$-module of type $(1,1)$. Then there exists a unique $\square_q$-module structure on $V$ such that the standard generators $x$ and $y$ act as $x_{0}$ and $x_{2}$ respectively. This $\square_q$-module is irreducible of type $1$.
\end{lem}

\begin{proof}
By \cite[\rm Theorem~10.4]{qtet}, there exists a $\square_q$-module $V$ such that the standard generators $x$ and $y$ act as $x_{0}$ and $x_{2}$ respectively. By this and since the $U_q^+$-module $V$ is irreducible, the $\square_q$-module $V$ is irreducible. By Lemma {\ref {lem:88}} and since the $U_q^+$-module $V$ is of type $(1,1)$, the $\square_q$-module $V$ is of type $1$. It suffices to show this $\square_q$-module structure on $V$ is unique. For each generator $x_i$ of $\square_q$ the action on $V$ is determined by the decomposition $[i,i+1]$. By Theorem {\ref {thm:93}} the decomposition $[i,i+1]$ is determined by the flags $[i]$ and $[i+1]$. Therefore our $\square_q$-module structure on $V$ is determined by the flags $[h]$, $h\in\mathbb{Z}_4$. By construction the flags $[0]$ and $[1]$ are determined by the decomposition $[0,1]$ and hence by the action of $x$ on $V$. Similarly the flags $[2]$ and $[3]$ are determined by the decomposition $[2,3]$ and hence by the action of $y$ on $V$. Therefore the given $\square_q$-module structure on $V$ is determined by the action of $x$ and $y$ on $V$, so this $\square_q$-module structure is unique. The result follows.
\end{proof}

\begin{thm}
\label{rmk:10.4}
The map $\kappa^\sharp$ gives a bijection between the following two sets:
\begin{enumerate}
\item[\rm(i)] the isomorphism classes of finite-dimensional irreducible $\square_q$-modules of type $1$;
\item[\rm(ii)] the isomorphism classes of NonNil finite-dimensional irreducible $U_q^+$-modules of type $(1,1)$.
\end{enumerate}
\end{thm}
\begin{proof}
By Lemmas {\ref{thm:10.2}} and {\ref {thm:10.3}}.
\end{proof}

Recall the map $\psi$ from above Lemma {\ref {lem:1}}.
\begin{thm}
\label{thm:101}
The map $\psi^\sharp$ gives a bijection between the following two sets:
\begin{enumerate}
\item[\rm(i)] the isomorphism classes of finite-dimensional irreducible $U_q(L(\mathfrak{sl}_2))$-modules $V$ of type $1$ such that $P_V(1)\ne 0$;
\item[\rm(ii)] the isomorphism classes of finite-dimensional irreducible $\square_q$-modules of type $1$.
\end{enumerate}
\end{thm}
\begin{proof}
The result follows from (\ref {k}) along with Theorems {\ref {thm:68}} and {\ref {rmk:10.4}}.
\end{proof}

Let $V$ denote a finite-dimensional irreducible $\square_q$-module of type $1$. Via Theorem {\ref{thm:101}}, the vector space $V$ becomes a finite-dimensional irreducible $U_q(L(\mathfrak{sl}_2))$-module of type $1$ such that $P_V(1)\ne 0$. Via Theorem {\ref{rmk:10.4}}, the vector space $V$ becomes a NonNil finite-dimensional irreducible $U_q^+$-module of type $(1,1)$.
\begin{1def}
\label{def:104}
Let $V$ denote a finite-dimensional irreducible $\square_q$-module of type $1$. Let $P_V$ (resp. $Q_V$) denote the polynomial from Definition {\ref {def:52}} (resp. Definition {\ref {def:52b}}) associated with the $U_q(L(\mathfrak{sl}_2))$-module $V$ from Theorem {\ref{thm:101}}. Observe that the polynomial $P_V$ (resp. $Q_V$) is equal to the polynomial from Definition {\ref {driuq}} associated with the $U_q^+$-module $V$ from Theorem {\ref{rmk:10.4}}.
\end{1def}

\begin{proof}[Proof of Proposition \ref{thm:36}]
By Theorems {\ref {thm:57}} and {\ref {thm:101}}.
\end{proof}

We now give a detailed description of the finite-dimensional irreducible $\square_q$-modules of type $1$ and diameter $1$. Pick $0\ne a\in\mathbb{F}$. Via Theorem {\ref{thm:101}}, the $U_q(L(\mathfrak{sl}_2))$-module $V(1,a)$ from Lemma {\ref {eg1}} becomes a $\square_q$-module.
\begin{lem}
\label{eg4}
The $\square_q$-modules
\begin{equation}
\label{evae2}
V(1,a) \qquad \qquad a\in \mathbb{F}\backslash \{0,1\}
\end{equation}
are irreducible. With respect to the basis $v_0,v_1$ from Lemma {\ref {eg1}}, the $\square_q$ generators $\{x_i\}_{i\in\mathbb{Z}_4}$ are represented by the following matrices:
\begin{align}
\begin{aligned}
x_{0}:\,
\begin{bmatrix}
    q^{-1}       & (q-q^{-1})a^{-1} \\
    0       & q
\end{bmatrix},\qquad\quad
x_{2}:\,
\begin{bmatrix}
    q       & 0 \\
    q-q^{-1}       & q^{-1}
\end{bmatrix}.\label{matr2}\\
x_{1}:\,
\begin{bmatrix}
    q       & q^{-1}-q \\
    0       & q^{-1}
\end{bmatrix},\qquad\quad
x_{3}:\,
\begin{bmatrix}
    q^{-1}       & 0 \\
    (q^{-1}-q)a      & q
\end{bmatrix}.
\end{aligned}
\end{align}
Moreover $V(1,a)$ has type $1$ and diameter $1$. Every finite-dimensional irreducible $\square_q$-module of type $1$ and diameter $1$ is isomorphic to exactly one of the modules (\ref {evae2}).
\end{lem}
\begin{proof}
Pick $0\ne a\in\mathbb{F}$. By Lemma {\ref {eg3}} and Theorem {\ref {rmk:10.4}}, the $\square_q$-module $V(1,a)$ is irreducible if and only if $a\ne 1$. The matrices in (\ref {matr2}) are obtained from Lemma {\ref {lem:qeva}} and Theorem {\ref{thm:101}}. The remaining assertions follow by Lemma {\ref {eg1}} and Theorem {\ref{thm:101}}.
\end{proof}

\begin{lem}
\label{lem:998}
Pick $a\in \mathbb{F}\backslash \{0,1\}$. Let $V$ denote the $\square_q$-module $V(1,a)$ from Lemma {\ref {eg4}}. Then $P_V(z)=1-az$ and $Q_V(z)=1-a^{-1}z$.
\end{lem}
\begin{proof}
By Lemma {\ref {eg2}} and Definition {\ref {def:104}}.
\end{proof}

\section{Tridiagonal pairs}
Our next general goal is to prove Theorem {\ref {thm:37}}. To do this, it is convenient to bring in the notion of a tridiagonal pair. In this section we recall the definition of a tridiagonal pair and review its basic properties. We are mainly interested in a family of tridiagonal pairs said to have $q$-geometric type. Let $V$ denote a finite-dimensional irreducible $\square_q$-module of type $1$. Near the end of this section we show that for $i\in\mathbb{Z}_4$ the $\square_q$ generators $x_i,x_{i+2}$ act on $V$ as a $q$-geometric tridiagonal pair.

From now until the end of Theorem {\ref {thm:12.3}}, let $V$ denote a vector space over $\mathbb{F}$ with finite positive dimension.
\begin{1def}
\label{def:trid}
(See \cite[\rm Definition~1.1]{tri}.)
By a {\it tridiagonal pair} on $V$ we mean an ordered pair $A,A^*$ of elements in ${\rm End}(V)$ that satisfy (i)--(iv) below:
\begin{enumerate}
\item[\rm(i)] Each of $A,A^*$ is semisimple.
\item[\rm(ii)] There exists an ordering $\{V_i\}_{i=0}^d$ of the eigenspaces of $A$ such that
\begin{gather}
A^* V_i\subseteq V_{i-1}+V_i+V_{i+1}\qquad (0\le i\le d \label{equ:11.1}),
\end{gather}
where $V_{-1}=0$, $V_{d+1}=0$.
\item[\rm(iii)] There exists an ordering $\{V_i^*\}_{i=0}^\delta$ of the eigenspaces of $A^*$ such that
\begin{gather*}
A V_i^*\subseteq V_{i-1}^*+V_i^*+V_{i+1}^*\qquad (0\le i\le \delta),
\end{gather*}
where $V_{-1}^*=0$, $V_{\delta+1}^*=0$.
\item[\rm(iv)] There does not exist a subspace $W$ of $V$ such that $AW\subseteq W$, $A^* W\subseteq W$, $W\ne 0$, $W\ne V$.
\end{enumerate}
We say the tridiagonal pair $A,A^*$ is {\it over} $\mathbb{F}$.
\end{1def}

\begin{note}
According to a common notational convention, $A^*$ denotes the conjugate transpose of $A$. We are not using this convention. In a tridiagonal pair $A,A^*$ the linear transformations $A$ and $A^*$ are arbitrary subject to (i)--(iv) above.
\end{note}

We recall a few basic facts about tridiagonal pairs. Let $A,A^*$ denote a tridiagonal pair on $V$ and let $d,\delta$ be as in Definition {\ref {def:trid}}(ii), (iii). By \cite[\rm Lemma~4.5]{tri} we have $d=\delta$; we call this common value the {\it diameter} of $A,A^*$. By \cite[\rm Corollary~5.7]{tri}, for $0\le i\le d$ the subspaces $V_i,V_i^*$ have the same dimension; we denote this common dimension by $\rho_i$. By the construction $\rho_i\ne 0$. By \cite[\rm Corollary~5.7, 6.6]{tri}, the sequence $\{\rho_i\}_{i=0}^d$ is symmetric and unimodal; that is $\rho_i=\rho_{d-i}$ for $0\le i\le d$ and $\rho_{i-1}\le \rho_{i}$ for $1\le i\le d/2$. We call the sequence $\{\rho_i\}_{i=0}^d$ the {\it shape} of $A,A^*$. The pair $A,A^*$ is said to be {\it sharp} whenever $\rho_0=1$. By \cite[\rm Theorem~1.3]{stru} and since $\mathbb{F}$ is algebraically closed, the tridiagonal pair $A,A^*$ is sharp. An ordering of the eigenspaces of $A$ is called {\it standard} whenever it satisfies ({\ref {equ:11.1}}). We comment on the uniqueness of the standard ordering. Let $\{V_i\}_{i=0}^d$ denote a standard ordering of the eigenspaces of $A$. Then the ordering $\{V_{d-i}\}_{i=0}^d$ is standard and no further ordering is standard. An ordering of the eigenvalues of $A$ is called {\it standard} whenever the corresponding ordering of the eigenspaces of $A$ is standard. Similar comments apply to $A^*$. Let $\{V_i\}_{i=0}^d$ (resp. $\{V_i^*\}_{i=0}^d$) denote a standard ordering of the eigenspaces of $A$ (resp. $A^*$). For $0\le i\le d$ let $\theta_i$ (resp. $\theta_i^*$) denote the eigenvalue of $A$ (resp. $A^*$) that corresponds to the eigenspace $V_i$ (resp. $V_i^*$). By construction the orderings $\{\theta_i\}_{i=0}^d$ and $\{\theta_i^*\}_{i=0}^d$ are standard. By \cite[\rm Theorem~4.6]{tri}, for $0\le i\le d$ the subspace $V_0^*$ is invariant under
\begin{gather*}
(A^*-\theta_1^*I)(A^*-\theta_2^*I)\dots (A^*-\theta_i^*I)(A-\theta_{i-1}I)\dots (A-\theta_{1}I)(A-\theta_{0}I).
\end{gather*}
Let $\zeta_i$ denote the corresponding eigenvalue. Note that $\zeta_0=1$. We call the sequence $\{\zeta_i\}_{i=0}^d$ the {\it split sequence} of $A,A^*$ with respect to the ordering $(\{\theta_i\}_{i=0}^d,\{\theta_i^*\}_{i=0}^d)$.

In \cite{cla} the tridiagonal pairs over $\mathbb{F}$ are classified up to isomorphism. In this classification there are several cases; the \enquote{most general} case is called type I. We will recall the type I case shortly.
\begin{lem}
\label{beta}
\rm
(See \cite[\rm Theorem~11.1]{tri}.)
\it
Let $A,A^*$ denote a tridiagonal pair on $V$. Let $\{\theta_i\}_{i=0}^d$ (resp. $\{\theta_i^*\}_{i=0}^d$) denote a standard ordering of the eigenvalues of $A$ (resp. $A^*$). Then the expressions
\begin{gather}
\frac{\theta_{i-2}-\theta_{i+1}}{\theta_{i-1}-\theta_i},\qquad \qquad \frac{\theta_{i-2}^*-\theta_{i+1}^*}{\theta_{i-1}^* -\theta_i^*} \label{equ10}
\end{gather}
are equal and independent of $i$ for $2\le i\le d-1$.
\end{lem}

\begin{1def}
(See \cite[\rm Definition~4.3]{dri}.)
\label{beta1}
Let $A,A^*$ denote a tridiagonal pair on $V$. We associate with $A,A^*$ a scalar $\beta\in\mathbb{F}$ as follows. If the diameter $d\ge 3$ let $\beta+1$ denote the common value of (\ref {equ10}). If $d\le 2$ let $\beta$ denote any scalar in $\mathbb{F}$. We call $\beta$ a {\it base} of $A,A^*$. By construction, for $d\ge 3$ the tridiagonal pairs $A,A^*$ and $A^*,A$ have the same base. For $d\le 2$, we always chose the bases such that $A,A^*$ and $A^*,A$ have the same base.
\end{1def}

\begin{1def}
\label{type}
Let $A,A^*$ denote a tridiagonal pair on $V$. We say that the base of $A,A^*$ has {\it type I} whenever this base is not equal to $2$ or $-2$.
\end{1def}

\begin{lem}
\label{thet}
\rm
(See \cite[\rm Theorem~11.2]{tri}.)
\it
Let $A,A^*$ denote a tridiagonal pair on $V$ that has diameter $d$. Let $\{\theta_i\}_{i=0}^d$ (resp. $\{\theta_i^*\}_{i=0}^d$) denote a standard ordering of the eigenvalues of $A$ (resp. $A^*$). Let $\beta$ denote a base of $A,A^*$ and assume that $\beta$ has type I. Fix a nonzero $t\in \mathbb{F}$ such that $t^2+t^{-2}=\beta$. Then there exists a sequence of scalars $a,b,c,a^*,b^*,c^*$ taken from $\mathbb{F}$ such that
\begin{gather*}
\theta_i=a+bt^{2i-d}+ct^{d-2i},\\
\theta_i^*=a^*+b^*t^{2i-d}+c^*t^{d-2i}
\end{gather*}
for $0\le i\le d$. This sequence is uniquely determined by $t$ provided $d\ge 2$.
\end{lem}
With reference to Lemma {\ref {thet}}, we have $t^4\ne 1$ since $t^2+t^{-2}=\beta$ and $\beta\ne \pm 2$. More generally, we have the following result.

\begin{lem}
\rm
(See \cite[\rm Theorem~11.2]{tri}.)
\it
With reference to Lemma {\ref {thet}}, we have $t^{2i}\ne 1$ for $1\le i\le d$.
\end{lem}

We now recall from \cite{dri} the Drinfel'd polynomial for a tridiagonal pair. We will restrict our attention to the case of type I. Until the end of Lemma {\ref {star}} assume that $d\ge 2$. With reference to Lemma {\ref {thet}}, for $1\le i\le d$ define $p_i\in\mathbb{F}[z]$ by
\begin{gather*}
p_i=(t^i-t^{-i})^2(bb^*t^{2i-2d}+cc^*t^{2d-2i}-z).
\end{gather*}
Define $P_{A,A^*}\in\mathbb{F}[z]$ by
\begin{gather}
P_{A,A^*}=e\sum_{i=0}^d \zeta_ip_{i+1}p_{i+2}\dots p_d, \label{dd}
\end{gather}
where $e=(-1)^d([d]_t^!)^{-2}(t-t^{-1})^{-2d}$ and $\{\zeta_i\}_{i=0}^d$ is the split sequence of $A,A^*$ with respect to the ordering $(\{\theta_i\}_{i=0}^d,\{\theta_i^*\}_{i=0}^d)$. By \cite[\rm Theorem~12.12]{dri}, the polynomial $P_{A,A^*}$ is independent of the choice of the standard orderings for the eigenvalues of $A$ and $A^*$. We call $P_{A,A^*}$ the {\it Drinfel'd polynomial} for the tridiagonal pair $A,A^*$.

\begin{note}
Our above definition of the Drinfel'd polynomial for the tridiagonal pair $A,A^*$ is slightly different from the one in \cite[\rm Definition~9.3]{dri}. The Drinfel'd polynomial for $A,A^*$ in \cite[\rm Definition~9.3]{dri} is equal to the right-hand side of (\ref {dd}) except that the factor $e$ is missing.
\end{note}

\begin{lem}
\label{star}
\rm
(See \cite[\rm Theorem~12.12]{dri}.)
\it
Let $A,A^*$ denote a tridiagonal pair on $V$ that has diameter at least $2$ and a base with type I. Then $P_{A,A^*}=P_{A^*,A}$.
\end{lem}

\begin{1def}
\label{def:12.1}
(See \cite[\rm Definition~2.6]{non}.)
Let $d\in\mathbb{N}$ and let $A,A^*$ denote a tridiagonal pair on $V$ that has diameter $d$. Fix a nonzero $t\in \mathbb{F}$ such that $t^{4}\ne 1$. Then $A,A^*$ is called $t$-{\it geometric} whenever the sequence $\{t^{d-2i}\}_{i=0}^d$ is a standard ordering of the eigenvalues of $A$ and a standard ordering of the eigenvalues of $A^*$.
\end{1def}

Let $A,A^*$ denote a $t$-geometric tridiagonal pair on $V$. Define $\beta=t^2+t^{-2}$. By Definitions {\ref {beta1}} and {\ref {def:12.1}}, the scalar $\beta$ is a base of $A,A^*$. Observe that $\beta$ has type I.
\begin{thm}
\label{thm:12.2}
\rm
(See \cite[\rm Lemma~4.8]{non}.)
\it
\label{thm:12.3}
For $A,A^*\in {\rm End}(V)$ the following are equivalent:
\begin{enumerate}
\item[\rm(i)] the pair $A,A^*$ acts on $V$ as a $q$-geometric tridiagonal pair;
\item[\rm(ii)] the vector space $V$ is a NonNil irreducible $U_q^+$-module of type $(1,1)$ on which the standard generators $x,y$ act as $A,A^*$ respectively.
\end{enumerate}
\end{thm}

For the duration of this paragraph assume that $d\ge 2$. Let $V$ denote a NonNil finite-dimensional irreducible $U_q^+$-module of type $(1,1)$ and diameter $d$. By Theorem {\ref {thm:12.3}}, the pair $x,y$ acts on $V$ as a $q$-geometric tridiagonal pair. The polynomials $P_V$, $Q_V$ and $P_{x,y}$ are associated with $V$. By Lemma {\ref {lem:54}} and Definition {\ref {driuq}}, the polynomials $P_V$, $Q_V$ are partners. Therefore $z^dQ_V(z^{-1})$ is a scalar multiple of $P_V$.
\begin{lem}
\label{lem:100}
Assume that $d\ge 2$. Let $V$ denote a NonNil finite-dimensional irreducible $U_q^+$-module of type $(1,1)$ and diameter $d$. Then
\begin{gather*}
P_{x,y}(z)=z^dQ_V(z^{-1}).
\end{gather*}
\end{lem}
\begin{proof}
By Lemma {\ref {star}} it suffices to show that $P_{y,x}(z)=z^dQ_V(z^{-1})$. By Theorem {\ref{rmk:10.4}} the vector space $V$ becomes a $\square_q$-module on which $x=x_0$ and $y=x_2$. This $\square_q$-module $V$ is irreducible of type $1$. let $\{W_i\}_{i=0}^d$ (resp. $\{W_i^*\}_{i=0}^d$) denote the decomposition $[2,3]$ (resp. $[1,0]$) of the $\square_q$-module $V$. Observe that $\{W_i\}_{i=0}^d$ (resp. $\{W_i^*\}_{i=0}^d$) is a standard ordering of the eigenspaces of $y$ (resp. $x$). With respect to this ordering, we have
$a=a^*=0$, $b=c^*=0$, $c=b^*=1$ in view of Lemma {\ref {thet}}. By Theorem {\ref {thm:101}} the $\square_q$-module $V$ becomes a $U_q(L(\mathfrak{sl}_2))$-module on which $x_0=X_{01}$ and $x_2=X_{23}$.
This $U_q(L(\mathfrak{sl}_2))$-module $V$ is irreducible of type $1$ and $P_V(1)\ne 0$. By Definition {\ref {def:lp1}} and ({\ref {42}}), ({\ref {44}}) we have
\begin{gather}
(e_0^-)^i(e_1^-)^i=(q-q^{-1})^{-2i}(X_{01}-X_{31})^i(X_{23}-X_{13})^i\qquad (i\in\mathbb{N}).\label{10121}
\end{gather}
Let $\{U_i\}_{i=0}^{d}$ denote the weight space decomposition of the $U_q(L(\mathfrak{sl}_2))$-module $V$ from Lemma {\ref {lem:dec}}. Observe that $U_0=W_0^*$. By Definition {\ref {def:51b}}, the left-hand side of (\ref {10121}) acts on $W_0^*$ as $\mu_i I$. By ({\ref {42}}), ({\ref {44}}) and Lemma {\ref {lem:dec}}, we have $(X_{01}-X_{31})U_i\subseteq U_{i-1}$ and $(X_{23}-X_{13})U_i\subseteq U_{i+1}$ for $0\le i\le d$. By ({\ref {44}}) and Lemma {\ref {lem:dec}} along with the fact that $x=X_{01}$ and $y=X_{23}$ on $V$, we have $X_{01}-X_{31}=x-q^{2i-d}I$ and $X_{23}-X_{13}=y-q^{d-2i}I$ on $U_i$ for $0\le i\le d$. By these comments, the following holds on $W_0^*$ for $0\le i\le d$:
\begin{gather}
\begin{split}
(X_{01}-X_{31})^i(X_{23}-X_{13})^i &=(x-q^{2-d}I)(x-q^{4-d}I)\dots (x-q^{2i-d}I)\\
&(y-q^{d-2i+2}I)\dots (y-q^{d-2}I)(y-q^{d}I).\label{10122}
\end{split}
\end{gather}
Let $\{\zeta_i\}_{i=0}^d$ denote the split sequence for the tridiagonal pair $y,x$ with respect to the orderings $\{W_i\}_{i=0}^d$, $\{W_i^*\}_{i=0}^d$. By the definition of split sequence above Lemma {\ref {beta}}, the right-hand side of (\ref {10122}) acts on $W_0^*$ as $\zeta_i I$.
By the above comments, we have $\zeta_i=(q-q^{-1})^{2i}\mu_i$. By this and Definition {\ref {def:52b}} along with (\ref {dd}), we routinely obtain $P_{y,x}(z)=z^dQ_V(z^{-1})$. The result follows.
\end{proof}

Let $V$ denote a finite-dimensional irreducible $\square_q$-module of type $1$. Via Theorem {\ref{rmk:10.4}}, the vector space $V$ becomes a NonNil finite-dimensional irreducible $U_q^+$-module of type $(1,1)$.
\begin{lem}
\label{lem:13.8}
Assume that $d\ge 2$. Let $V$ denote a finite-dimensional irreducible $\square_q$-module of type $1$ and diameter $d$. Then
\begin{gather*}
P_{x_0,x_2}(z)=z^dQ_V(z^{-1}).
\end{gather*}
\end{lem}
\begin{proof}
Apply Lemma {\ref {lem:100}} with $x=x_0$ and $y=x_2$.
\end{proof}

\begin{lem}
\label{lem:13.5}
Let $V$ denote a finite-dimensional irreducible $\square_q$-module of type $1$. Then for $i\in\mathbb{Z}_4$ the $\square_q$ generators $x_i,x_{i+2}$ act on $V$ as a $q$-geometric tridiagonal pair.
\end{lem}
\begin{proof}
By Theorems {\ref {rmk:10.4}} and {\ref {thm:12.2}}, the pair $x_0,x_2$ acts on $V$ as a $q$-geometric tridiagonal pair. Applying this to $^{\rho}V$, $^{\rho^2}V$, $^{\rho^3}V$ we obtain the remaining assertions.
\end{proof}

\section{The proof of Theorem {\ref {thm:37}}}
Let $V$ denote a finite-dimensional irreducible $\square_q$-module of type $1$. In the previous section we showed that for $i\in\mathbb{Z}_4$ the $\square_q$ generators $x_i,x_{i+2}$ act on $V$ as a $q$-geometric tridiagonal pair. In this section we will describe the relations among the polynomials
\begin{gather*}
P_{x_i,x_{i+2}},\qquad  P_{^{\rho^i}V} ,\qquad Q_{^{\rho^i}V} \qquad (i\in\mathbb{Z}_4).
\end{gather*}
Using these results, we prove Theorem {\ref {thm:37}}.

\begin{lem}
\label{lem:13.7}
Assume that $d\ge 2$. Let $V$ denote a finite-dimensional irreducible $\square_q$-module of type $1$ and diameter $d$. Then
\begin{gather*}
P_{x_1,x_3}(z)=z^dP_V(z^{-1}).
\end{gather*}
\end{lem}
\begin{proof}
By Lemma {\ref {star}} it suffices to show that $P_{x_3,x_1}(z)=z^dP_V(z^{-1})$. Let $\{W_i\}_{i=0}^d$ (resp. $\{W_i^*\}_{i=0}^d$) denote the decomposition $[0,3]$ (resp. $[1,2]$) of the $\square_q$-module $V$. Observe that $(\{W_i\}_{i=0}^d$  (resp. $\{W_i^*\}_{i=0}^d$) is a standard ordering of the eigenspaces of $x_3$ (resp. $x_1$). With respect to this ordering, we have $a=a^*=0$, $b=c^*=1$, $c=b^*=0$ in view of Lemma {\ref {thet}}. By Theorem {\ref {thm:101}} the $\square_q$-module $V$ becomes a $U_q(L(\mathfrak{sl}_2))$-module on which $x_3=X_{30}$ and $x_1=X_{12}$.
This $U_q(L(\mathfrak{sl}_2))$-module $V$ is irreducible of type $1$ and $P_V(1)\ne 0$. By Definition {\ref {def:lp1}} and ({\ref {43}}), ({\ref {44}}) we have
\begin{gather}
(e_1^+)^i(e_0^+)^i=(q-q^{-1})^{-2i}(X_{12}-X_{13})^i(X_{30}-X_{31})^i \qquad (i\in\mathbb{N}).\label{1111}
\end{gather}
Let $\{U_i\}_{i=0}^{d}$ denote the weight space decomposition of the $U_q(L(\mathfrak{sl}_2))$-module $V$ from Lemma {\ref {lem:dec}}. Observe that $U_0=W_0^*$. By Definition {\ref {def:51}} the right-hand side of (\ref {1111}) acts on $W_0^*$ as $\sigma_i I$. By ({\ref {43}}), ({\ref {44}}) and Lemma {\ref {lem:dec}}, we have $(X_{12}-X_{13})U_i\subseteq U_{i-1}$ and $(X_{30}-X_{31})U_i\subseteq U_{i+1}$ for $0\le i\le d$. By ({\ref {44}}) and Lemma {\ref {lem:dec}} along with the fact that $x_3=X_{30}$ and $x_1=X_{12}$ on $V$, we have $X_{12}-X_{13}=x_1-q^{d-2i}I$ and $X_{30}-X_{31}=x_3-q^{2i-d}I$ on $U_i$ for $0\le i\le d$. By these comments, the following holds on $W_0^*$ for $0\le i\le d$:
\begin{gather}
\begin{split}
(X_{12}-X_{13})^i(X_{30}-X_{31})^i&=(x_1-q^{d-2}I)(x_1-q^{d-4}I)\dots (x_1-q^{d-2i}I)\\
&(x_3-q^{2i-2-d}I)\dots (x_3-q^{2-d}I)(x_3-q^{-d}I).\label{1112}
\end{split}
\end{gather}
Let $\{\zeta_i\}_{i=0}^d$ denote the split sequence for the tridiagonal pair $x_3,x_1$ with respect to the orderings $\{W_i\}_{i=0}^d$, $\{W_i^*\}_{i=0}^d$. By the definition of split sequence above Lemma {\ref {beta}}, the right-hand side of (\ref {1112}) acts on $W_0^*$ as $\zeta_i I$. By the above comments, we have $\zeta_i=(q-q^{-1})^{2i}\sigma_i$. By this and Definition {\ref {def:52}} along with (\ref {dd}), we routinely obtain $P_{x_3,x_1}(z)=z^dP_V(z^{-1})$. The result follows.
\end{proof}

\begin{pro}
\label{table}
Assume that $d\ge 2$. Let $V$ denote a finite-dimensional irreducible $\square_q$-module of type $1$ and diameter $d$. Then the following six polynomials coincide:
\begin{gather*}
P_{x_0,x_2}(z),\qquad z^dP_{^{\rho}V}(z^{-1}),\qquad z^dQ_V(z^{-1}),\\
P_{x_2,x_0}(z),\qquad z^dP_{^{\rho^3}V}(z^{-1}),\qquad z^dQ_{^{\rho^2} V}(z^{-1}).
\end{gather*}
Moreover the the following six polynomials coincide:
\begin{gather*}
P_{x_1,x_3}(z),\qquad z^dP_{V}(z^{-1}),\qquad z^dQ_{^{\rho}V}(z^{-1}),\\
P_{x_3,x_1}(z),\qquad z^dP_{^{\rho^2}V}(z^{-1}),\qquad z^dQ_{^{\rho^3} V}(z^{-1}).
\end{gather*}
\end{pro}
\begin{proof}
Apply Lemmas {\ref {lem:13.8}} and {\ref {lem:13.7}} to $^{\rho^i}V$ for $i\in\mathbb{Z}_4$ and use Lemma {\ref {star}}.
\end{proof}

\begin{cor}
\label{cor}
Let $V$ denote a finite-dimensional irreducible $\square_q$-module of type $1$. Then $P_{^{\rho}V}=Q_V$.
\end{cor}
\begin{proof}
Let $d$ denote the diameter of the $\square_q$-module $V$. First assume that $d=0$, then $P_{^{\rho}V}=1=Q_V$ by Definitions {\ref {def:52}}, {\ref {def:52b}} and {\ref {def:104}}. Next assume that $d=1$. By construction, the $\square_q$-module $^{\rho}V$ is irreducible of type $1$ and diameter $1$. Via Theorem {\ref {thm:101}}, the $\square_q$-module $V$ (resp. $^\rho V$) becomes a finite-dimensional irreducible $U_q(L(\mathfrak{sl}_2))$-module of type $1$ and diameter $1$. Let $a$ (resp. $b$) denote the evaluation parameter of the $U_q(L(\mathfrak{sl}_2))$-module $V$ (resp. $^\rho V$). By Lemma {\ref {lem:998}}, we have $P_{^\rho V}(z)=1-bz$ and $Q_V(z)=1-a^{-1}z$. To show that $P_{^{\rho}V}=Q_V$, it suffices to show that $b=a^{-1}$. By construction, the action of $X_{01}$ (resp. $X_{23}$) on $^\rho V$ is equal to the action of $X_{12}$ (resp. $X_{30}$) on $V$. Therefore ${\rm{tr}}(X_{01}X_{23})$ on $^\rho V$ is equal to ${\rm{tr}}(X_{12}X_{30})$ on $V$. By this, (\ref {tr1})--(\ref {tr2}) and $q^2\ne 1$ we have $b=a^{-1}$. Therefore $P_{^{\rho}V}=Q_V$. Next assume that $d\ge 2$. Then $P_{^{\rho}V}=Q_V$ by Proposition {\ref {table}}. We have shown that $P_{^{\rho}V}=Q_V$ for every value of $d$. The result follows.
\end{proof}

\begin{proof}[Proof of Theorem \ref{thm:37}]
Combine Lemma {\ref {lem:54}} and Corollary {\ref {cor}}.
\end{proof}

\begin{proof}[Proof of Corollary \ref{cor:38}]
Apply Theorem {\ref {thm:37}} twice to $V$, and use Proposition {\ref {thm:36}}.
\end{proof}

\section{The $q$-tetrahedron algebra $\boxtimes_q$}
Recall the result Theorem {\ref {thm:37}} about the algebra $\square_q$. In this section we obtain an analogous result for $\boxtimes_q$.
\begin{1def}
(See \cite[\rm Definition~6.1]{qtet}.)
\label{qt}
Let $\boxtimes_q$ denote the $\mathbb{F}$-algebra with generators
\begin{gather*}
\{x_{ij}\mid i,j\in\mathbb{Z}_4, j-i=1\, or \,j-i=2\}
\end{gather*}
and the following relations:
\begin{enumerate}
\item[\rm(i)] For $i,j\in\mathbb{Z}_4$ and $j-i=2$,
\begin{gather*}
x_{ij}x_{ji}=1.
\end{gather*}
\item[\rm(ii)] For $h,i,j\in\mathbb{Z}_4$ such that the pair $(i-h,j-i)$ is one of $(1,1),(1,2),(2,1)$,
\begin{gather*}
\frac{qx_{hi}x_{ij}-q^{-1}x_{ij}x_{hi}}{q-q^{-1}}=1.
\end{gather*}
\item[\rm(iii)] For $h,i,j,k\in\mathbb{Z}_4$ such that $i-h=j-i=k-j=1$,
\begin{gather*}
x_{hi}^3x_{jk}-[3]_qx_{hi}^2x_{jk}x_{hi}+[3]_qx_{hi}x_{jk}x_{hi}^2-x_{jk}x_{hi}^3=0.
\end{gather*}
\end{enumerate}
We call $\boxtimes_q$ the {\it $q$-tetrahedron algebra}.
\end{1def}

We will use the following automorphism of $\boxtimes_q$.
\begin{lem}
\label{lem:776}
\rm
(See \cite[\rm Lemma~6.3]{qtet}.)
\it
There exists an automorphism $\rho$ of $\boxtimes_q$ that sends each generator $x_{ij}$ to $x_{i+1,j+1}$. Moreover $\rho^4=1$.
\end{lem}

\begin{note}
We use the same symbol $\rho$ for an automorphism of $\square_q$ (in Lemma {\ref {lem:rho}}) and $\boxtimes_q$ (in Lemma {\ref {lem:776}}). As we proceed, it should be clear from the context which algebra is being discussed.
\end{note}

Comparing the relations in Lemma {\ref {lem:lp2}} with the relations in Definition {\ref {qt}}, we obtain an $\mathbb{F}$-algebra homomorphism $\eta: U_q(L(\mathfrak{sl}_2)) \to \boxtimes_q$ that sends $X_{13}\mapsto x_{13}$, $X_{31}\mapsto x_{31}$ and $X_{i,i+1}\mapsto x_{i,i+1}$ for $i\in\mathbb{Z}_4$.

\begin{lem}
\rm
(See \cite[Propositions~4.3]{miki}.)
\label{lem:3}
\it
The above homomorphism $\eta$ is injective.
\end{lem}

We next recall some facts about finite-dimensional irreducible $\boxtimes_q$-modules.

\begin{lem}
\label{lem:74}
\rm
(See \cite[\rm Theorem~12.3]{qtet}.)
\it
Let $V$ denote a finite-dimensional irreducible $\boxtimes_q$-module. Then each generator $x_{ij}$ of $\boxtimes_q$ is semisimple on $V$. Moreover there exist $d\in\mathbb{N}$ and $\gamma\in\{1,-1\}$ such that for each generator $x_{ij}$ the set of distinct eigenvalues of $x_{ij}$ on $V$ is $\{\gamma q^{d-2i}\}_{i=0}^d$.
\end{lem}

\begin{1def}
(See \cite[\rm Definition~12.4]{qtet}.)
Let $V$ denote a finite-dimensional irreducible $\boxtimes_q$-module. By the {\it diameter} of $V$ we mean the scalar $d$ from Lemma {\ref {lem:74}}. By the {\it type} of $V$ we mean the scalar $\gamma$ from Lemma {\ref {lem:74}}.
\end{1def}

Recall the map $\kappa$ from above Lemma {\ref {lem:2}} and the map $\psi$ from above Lemma {\ref {lem:1}}. Consider the composition
\begin{equation*}
\eta\circ\psi\circ\kappa:
\begin{tikzcd}
 U_q^+\arrow{r}{\kappa} & \square_q \arrow{r}{\psi}& U_q(L(\mathfrak{sl}_2))\arrow{r}{\eta}& \boxtimes_q.
\end{tikzcd}
\end{equation*}

\begin{thm}
\rm
(See \cite[\rm Theorem~10.3, 10.4]{qtet}.)
\label{thm:78}
\it
The map $ (\eta\circ\psi\circ\kappa)^\sharp$ gives a bijection between the following two sets:
\begin{enumerate}
\item[\rm(i)] the isomorphism classes of finite-dimensional irreducible $\boxtimes_q$-modules of type $1$;
\item[\rm(ii)] the isomorphism classes of NonNil finite-dimensional irreducible $U_q^+$-modules of type $(1,1)$.
\end{enumerate}
\end{thm}

\begin{thm}
\label{thm:79}
The map $\eta^\sharp$ gives a bijection between the following two sets:
\begin{enumerate}
\item[\rm(i)] the isomorphism classes of finite-dimensional irreducible $\boxtimes_q$-modules of type $1$;
\item[\rm(ii)] the isomorphism classes of finite-dimensional irreducible $U_q(L(\mathfrak{sl}_2))$-modules $V$ of type $1$ such that $P_V(1)\ne 0$.
\end{enumerate}
\end{thm}
\begin{proof}
Observe that $(\eta\circ\psi\circ\kappa)^\sharp=(\psi\circ\kappa)^\sharp\circ\eta^\sharp$. The result follows from this along with Theorems {\ref {thm:68}}, {\ref {thm:78}}.
\end{proof}

\begin{thm}
\label{thm:eta}
The map $(\eta\circ\psi)^\sharp$ gives a bijection between the following two sets:
\begin{enumerate}
\item[\rm(i)] the isomorphism classes of finite-dimensional irreducible $\boxtimes_q$-modules of type $1$;
\item[\rm(ii)] the isomorphism classes of finite-dimensional irreducible $\square_q$-modules of type $1$.
\end{enumerate}
\end{thm}
\begin{proof}
Observe that $(\eta\circ\psi)^\sharp=\psi^\sharp\circ \eta^\sharp$. The result follows from this along with Theorems {\ref {thm:101}}, {\ref {thm:79}}.
\end{proof}

Let $V$ denote a finite-dimensional irreducible $\boxtimes_q$-module of type $1$. Via Theorem {\ref{thm:79}}, the vector space $V$ becomes a finite-dimensional irreducible $U_q(L(\mathfrak{sl}_2))$-module of type $1$ such that $P_V(1)\ne 0$. Via Theorem {\ref {thm:78}}, the vector space $V$ becomes a NonNil finite-dimensional irreducible $U_q^+$-module of type $(1,1)$. Via Theorem {\ref {thm:eta}}, the vector space $V$ becomes a finite-dimensional irreducible $\square_q$-module of type $1$.
\begin{1def}
\label{def:7104}
Let $V$ denote a finite-dimensional irreducible $\boxtimes_q$-module of type $1$. Let $P_V$ denote the polynomial from Definition {\ref {def:52}} associated with the $U_q(L(\mathfrak{sl}_2))$-module $V$ from Theorem {\ref{thm:79}}. Observe that $P_V$ is equal to the polynomial from Definition {\ref {driuq}} associated with the $U_q^+$-module $V$ from Theorem {\ref {thm:78}} and the polynomial from Definition {\ref{def:104}} associated with the $\square_q$-module $V$ from Theorem {\ref {thm:eta}}.
\end{1def}

\begin{pro}
\label{thm:736}
The map $V\mapsto P_V$ induces a bijection between the following two sets:
\begin{enumerate}
\item[\rm(i)] the isomorphism classes of finite-dimensional irreducible $\boxtimes_q$-modules of type $1$;
\item[\rm(ii)] the polynomials in $\mathbb{F}[z]$ that have constant coefficient $1$ and do not vanish at $z=1$.
\end{enumerate}
\end{pro}
\begin{proof}
By Theorems {\ref {thm:57}} and {\ref {thm:79}}.
\end{proof}

Combining Theorems {\ref {rmk:10.4}}, {\ref {thm:101}}, {\ref {thm:12.2}}, {\ref {thm:79}}, {\ref {thm:736}}, we obtain a bijection between any two of the following sets:
\begin{enumerate}
\item[\rm(i)] the isomorphism classes of $q$-geometric tridiagonal pairs;
\item[\rm(ii)] the isomorphism classes of NonNil finite-dimensional irreducible $U_q^+$-modules of type $(1,1)$;
\item[\rm(iii)] the isomorphism classes of finite-dimensional irreducible $U_q(L(\mathfrak{sl}_2))$-modules $V$ of type $1$ such that $P_V(1)\ne 0$;
\item[\rm(iv)] the isomorphism classes of finite-dimensional irreducible $\boxtimes_q$-modules of type $1$;
\item[\rm(v)] the polynomials in $\mathbb{F}[z]$ that have constant coefficient $1$ and do not vanish at $z=1$;
\item[\rm(vi)] the isomorphism classes of finite-dimensional irreducible $\square_q$-modules of type $1$.
\end{enumerate}
Let $V$ denote a finite-dimensional irreducible $U_q(L(\mathfrak{sl}_2))$-module of type $1$ and $P_V(1)\ne 0$. Via the above bijections, we view $V$ as an $\mathcal{A}$-module, where $\mathcal{A}$ is any of $U_q^+$, $\square_q$, $U_q(L(\mathfrak{sl}_2))$, $\boxtimes_q$. We call $P_V$ the {\it Drinfel'd polynomial} for the $\mathcal{A}$-module $V$.

The map $\eta\circ\psi$ is an injective $\mathbb{F}$-algebra homomorphism from $\square_q$ to $\boxtimes_q$. We have the automorphism $\rho$ for $\square_q$ and $\boxtimes_q$. We now explain how these maps are related.
\begin{lem}
\label{lem:com}
The following diagram commutes:
\[
\xymatrix{
\square_q \ar[d]_{\rho} \ar[r]^{\eta\circ\psi} & \boxtimes_q\ar[d]^\rho\\
\square_q \ar[r]_{\eta\circ\psi} & \boxtimes_q}
\]
In the above diagram the map $\psi$ is from above Lemma {\ref {lem:1}}, the map $\eta$ is from above Lemma {\ref {lem:3}}, the map $\rho$ on the left is from Lemma {\ref {lem:rho}} and the map $\rho$ on the right is from Lemma {\ref {lem:776}}.
\end{lem}
\begin{proof}
Use the definitions of the maps in question.
\end{proof}

\begin{thm}
\label{thm:737}
Let $V$ denote a finite-dimensional irreducible $\boxtimes_q$-module of type $1$. Then the Drinfel'd polynomials for $V$ and $^\rho V$ are partners in the sense of Definition {\ref {def:part}}.
\end{thm}
\begin{proof}
By Lemmas {\ref {thm:eta}}, {\ref {lem:com}} and Theorem {\ref {thm:37}}.
\end{proof}

\begin{cor}
\label{cor:738}
Let $V$ denote a finite-dimensional irreducible $\boxtimes_q$-module of type $1$. Then the $\boxtimes_q$-modules $V$ and $^{\rho^2} V$ are isomorphic.
\end{cor}
\begin{proof}
Apply Theorem {\ref {thm:737}} twice to $V$, and use Proposition {\ref {thm:736}}.
\end{proof}

\section{Acknowledgement}
This paper was written while the author was a graduate student at the University of
Wisconsin-Madison. The author would like to thank his advisor, Paul Terwilliger, for offering
many valuable ideas and suggestions.

\end{document}